\documentclass[12pt,reqno,a4paper]{amsart}
\usepackage{amsmath,amssymb,amsfonts,amscd}
\usepackage[mathscr]{eucal}
\usepackage{comment,textcomp}

\numberwithin{equation}{section}

\usepackage{hyperref}  

\date{25 June (8 Jule) 2018}

\author{Theodore Voronov}

\address{{School of Mathematics,  University of Manchester,    Manchester,   M13 9PL,  UK}}
\email{theodore.voronov@manchester.ac.uk}
\address{%
{Faculty of Physics, Tomsk State University, Tomsk, 634050, Russia}}

\title[Microformal geometry  and homotopy algebras]{Microformal geometry and homotopy algebras}

\newcommand{\new}[1]{{#1}} 

\newcommand{\LHS}{{LHS}}
\newcommand{\RHS}{{RHS}}

\newtheorem{theorem}{Theorem}

\newtheorem{lemma}{Lemma}
\newtheorem{proposition}{Proposition}
\newtheorem{corollary}{Corollary}

\newtheorem{conjecture}{Conjecture}

\theoremstyle{definition}
\newtheorem{definition}{Definition}
\newtheorem*{transflaw}{Transformation Law}
\newtheorem{example}{Example}
\newtheorem{remark}{Remark}

\def\co{\colon\thinspace}
\newcommand{\mathcall}{\EuScript}

\renewcommand{\leq}{\leqslant}
\renewcommand{\geq}{\geqslant}

\DeclareMathOperator{\Ber}{Ber}

 \DeclareMathOperator{\exph}{\exp_{\hbar}}
  \DeclareMathOperator{\lnh}{\ln_{\hbar}}
 
 \DeclareMathOperator{\sgn}{sgn}

 \DeclareMathOperator{\Ad}{Ad}

 \DeclareMathOperator{\SMan}{{\mathcall{SM}an}}
 
 \DeclareMathOperator{\EThick}{{{\mathcall{ET}hick}}}
 \DeclareMathOperator{\OThick}{{{\mathcall{OT}hick}}}

\DeclareMathOperator{\ad}{ad}

\DeclareMathOperator{\fun}{\mathit{C^{\infty}}}
\DeclareMathOperator{\funn}{\mathbf{C^{\infty}}}
\DeclareMathOperator{\pfunn}{\mathbf{\Pi\!C^{\infty}}}
\DeclareMathOperator{\funh}{\mathit{C^{\infty}_{\hbar}}}
\DeclareMathOperator{\funnh}{\mathbf{C^{\infty}_{\hbar}}}

\newcommand{\der}[2]{{\frac{\partial {#1}}{\partial {#2}}}}

\newcommand{\lder}[2]{{\partial {#1}/\partial {#2}}}
\newcommand{\dder}[3]{{\frac{\partial^2 {#1}}{\partial {#2}\partial {#3}}}}
\newcommand{\var}[2]{{\frac{\delta {#1}{ }}{\delta {#2}}}}

\newcommand{\R}[1]{{\mathbb R}^{#1}}
\newcommand{\RR}{\mathbb R}
\newcommand{\Z}{{\mathbb Z_{2}}}
\newcommand{\ZZ}{{\mathbb Z}}
\newcommand{\CC}{\mathbb C}

\newcommand{\p}{\partial}

\renewcommand{\a}{\alpha}

\newcommand{\e}{\varepsilon}

\renewcommand{\O}{\Omega}
\newcommand{\D}{\Delta}
\renewcommand{\o}{\omega}
\newcommand{\g}{{\gamma}}
\newcommand{\h}{\eta}
\renewcommand{\t}{\theta}
\newcommand{\F}{{\Phi}}

\newcommand{\la}{{\lambda}}

\renewcommand{\t}{\theta}

\newcommand{\ft}{{\tilde f}}

\newcommand{\kt}{{\tilde k}}

\newcommand{\itt}{{\tilde \imath}}

\newcommand{\at}{{\tilde a}}
\newcommand{\att}{{\tilde \alpha}}
\newcommand{\bt}{{\tilde b}}
\newcommand{\ct}{{\tilde c}}
\newcommand{\ut}{{\tilde u}}

\newcommand{\qt}{{\tilde q}}

\newcommand{\Dt}{{\tilde \Delta}}

\newcommand{\Abold}{{\mathbf{A}}}
\newcommand{\Bbold}{{\mathbf{B}}}

\newcommand{\abold}{{\mathbf{a}}}

\DeclareMathOperator{\Vect}{\mathrm{Vect}}

\DeclareMathOperator{\Mult}{\mathfrak{A}}

\newcommand{\lsch}{{[\![}}
\newcommand{\rsch}{{]\!]}}

\newcommand{\tto}{{\linethickness{2pt}
		  \,\begin{picture}(1,0)
                   \put(0,0.26){\line(1,0){0.95}}
                   \put(0,0){$\boldsymbol{\rightarrow}$}
                  \end{picture}
                  }\,
}

\newcommand{\oto}{{\linethickness{0.5pt}
		  \,\begin{picture}(1,0)
		  \put(0.07,0.175){\line(0,1){0.2}}
                   \put(-0.01,0){$\boldsymbol{\Rightarrow}$}
                  \end{picture}
                  }\,
}

\newcommand{\ttto}[1]{\stackrel{#1}{\vphantom{\rightrightarrows}\tto}}

\newcommand{\tttoq}[1]{\stackrel{#1}{\vphantom{\rightrightarrows}\ttoq}}

\newcommand{\fto}[1]{\stackrel{#1}{\vphantom{\rightrightarrows}\to}}

\newcommand{\kir}{\boldsymbol{\kappa}}

\newcommand{\vp}{\boldsymbol{v}}
\newcommand{\ep}{\boldsymbol{e}}

\DeclareMathOperator{\gr}{graph}
\newcommand{\op}{\text{op}}
\newcommand{\infto}{\rightsquigarrow}

\DeclareMathOperator{\ofun}{\mathit{OC_{\hbar}^{\infty}}}
\newcommand{\ttoq}{\tto_q}

\newcommand{\dbar}{{\,\mathchar '26 \mkern-11mu d}}
\newcommand{\Dbar}{{\mathchar '26 \mkern-11 mu D}}

\newcommand{\f}{{\varphi}} 

\begin{document}
\begin{abstract}
We extend the category of (super)manifolds and their smooth mappings by introducing a notion of  microformal or ``thick'' morphisms. They are formal canonical relations of a special form,   constructed with the help of formal power expansions in cotangent directions. The result is a formal category in the sense that the its composition law is also specified by a formal power series. A microformal morphism acts on functions by an operation of pullback, which is in general a nonlinear transformation. More precisely, it is a formal mapping of formal manifolds of even functions (bosonic fields), which has the property that its derivative for every function is a ring homomorphism.  This suggests an abstract notion of a ``nonlinear algebra homomorphism'' and the corresponding extension of the classical ``algebraic-functional'' duality. There is a parallel fermionic version of these constructions.

The obtained formalism provides a general construction of   $L_{\infty}$-morphisms for functions on homotopy Poisson  ($P_{\infty}$) or homotopy Schouten   ($S_{\infty}$) manifolds  as   pullbacks by   Poisson microformal morphisms.  We also show that the notion of the adjoint linear operator can be generalized to nonlinear operators as a microformal morphism.   By applying this to    $L_{\infty}$-algebroids, we show that an $L_{\infty}$-morphism of  $L_{\infty}$-algebroids induces an    $L_{\infty}$-morphism of the  ``homotopy Lie--Poisson''  brackets for functions on the dual vector bundles. We apply this construction to higher Koszul brackets on differential forms and to triangular    $L_{\infty}$-bialgebroids. We also develop a quantum version (for the bosonic case), whose relation with the classical version is like that of the Schr\"{o}dinger equation with the Hamilton--Jacobi equation.  We show that the nonlinear pullbacks by microformal morphisms are the limits at   $\hbar\to 0$ of certain  ``quantum pullbacks'', which are defined as special form Fourier integral operators.
\end{abstract}

%


\maketitle
\tableofcontents

\section*{Introduction}

\subsection{Generalization of pullbacks and homotopy brackets} Constructing $L_{\infty}$-morphi\-sms between $L_{\infty}$-algebras is  in general  a difficult task; in some cases a particular example of an $L_{\infty}$-morphism can represent a solution of a highly non-trivial problem such as Kon\-tsevich's construction~\cite{kontsevich:quant} of deformation quantization of Poisson manifolds.

One of the results of this paper  is a general method giving    $L_{\infty}$-morphisms for $L_{\infty}$-algebras of functions. This is based on a certain extension or `thickening' of the usual category of smooth manifolds or supermanifolds.

It is well known that the  duality  of  the  geometric (``functional'') and algebraic viewpoints (see, e.g.,~\cite{shaf:a}) plays an important role in many mathematical theories, sometimes as a heuristic principle, and sometimes in the form of precise statements and constructions, such as the Gelfand  duality or Grothendieck's theory of schemes. By the geometric viewpoint, we mean   a picture based on ``spaces'' (in one or another sense), and by the algebraic viewpoint, a picture based on algebras, treated as algebras of functions. Under this duality, maps of spaces correspond to algebra homomorphisms, so that to a map there corresponds   the   pullback of functions, $\f^*\co g \mapsto \f^*(g)=g\circ \f$, which is a linear map  preserving the multiplication, i.e., a homomorphism. In the present paper, we  give constructions leading to a  nonlinear generalization  of such a  duality.

We  construct two formal categories extending     the 
category of smooth (super)manifolds and smooth maps,
with the same set of objects. Morphisms  $\Phi$ in these formal categories, which we call \emph{microformal} or  \emph{thick morphisms}, still act on smooth functions  by a generalization of pullbacks. A key ingredient in the construction is an equation of the fixed point type, whose solution is obtained by iterations.
Pullbacks by thick morphisms  $\Phi^*$ are formal \emph{nonlinear  differential operators}, represented by perturbative series around  ordinary pullbacks combined with additive shifts.  Nonlinearity is the distinctive property of these new pullbacks.  Similar equations and perturbative series arise for the composition law of thick morphisms (which is therefore formal).

Because of the nonlinearity, we have to distinguish   functions that are odd or even parity in the sense of the  $\Z$-grading as they have different commutativity properties.
That is why   there are  two formal categories, so that morphisms in one of them  denoted $\EThick$  induce pullbacks of even functions   (``bosonic fields''), while morphisms in another one denoted  $\OThick$ induce pullbacks of odd functions (``fermionic fields''). They are obtained by parallel constructions. Each of them contains the semidirect product   category  $\SMan\rtimes \funn$ or  $\SMan\rtimes \pfunn$, respectively, as a closed subspace and can be regarded as its formal neighborhood. (Here $\SMan$ is the ordinary category of supermanifolds and $\funn$ or $\pfunn$ are the spaces  of even or odd functions, on which smooth maps act by pullbacks.) There are embedding and retraction functors $\SMan\rtimes \funn \rightleftarrows \EThick$ and $\SMan\rtimes \pfunn \rightleftarrows \OThick$\,.

`Nonlinear pullbacks' were   first introduced by us in~\cite{tv:nonlinearpullback}  for the purpose  of construction of  $L_{\infty}$-morphisms of  homotopy Poisson algebras of functions  (motivated by a problem for higher Koszul brackets~\cite{tv:higherpoisson}). Such an $L_{\infty}$-morphism  by definition should be a nonlinear  map  of   functional supermanifolds, so it certainly cannot be a usual pullback. The idea of the   construction of a   ``nonlinear pullback'' was inspired by the  cotangent philosophy  of Kirill Mackenzie~\cite{mackenzie:book2005}.
As we showed,  these newly defined pullbacks with respect to thick morphisms indeed give the desired solution for homotopy Poisson brackets. Namely, if a   thick morphism $\Phi$ is \emph{Poisson}, which means that the master Hamiltonians or multivector fields specifying   homotopy Schouten or Poisson structures are $\Phi$-related  (a condition expressed in coordinates by a Hamilton--Jacobi type equation), then the pullback   {map} $\Phi^*$ is an $L_{\infty}$-morphism of the algebras of functions.

\subsection{Nonlinear algebro-functional duality} As the pullback  with respect to a thick morphisms is a nonlinear transformation, it cannot be a ring homomorphism  in the ordinary sense. It turns out, however, that its   \emph{derivative} at each function    will be a ring homomorphism!
Besides that, in spite of the nonlinearity, the pullbacks  themselves exhibit some kind of 
duality similar to the classical case.  For ordinary smooth maps, it is known that the   pullbacks on functions determine a map completely; in particular, giving the pullbacks of coordinate functions is the same as specifying a  map   in coordinates. Similarly for   a thick morphism, although it is not sufficient to know the images  of  individual  coordinate functions, it is sufficient however to know the images of  their  linear combinations $\Phi^*[y^i c_i]$
with arbitrary parameters $c_i$.
Another example of such a  ``nonlinear extension''  from multiplicative generators is given by the   pushforward  of functions on the dual vector spaces or vector bundles  by a nonlinear bundle  map. We introduce it as the pullback   with respect to the `adjoint operator'\,---\,which, as  we show, can be defined  for a nonlinear map, but as a thick morphism rather than an ordinary map; as we show, on vectors or on sections of the original bundle  this pushforward agrees with a given nonlinear mapping.

Algebraic properties of nonlinear pullbacks suggest the following abstract framework. For algebras $A$ and $B$, define a \emph{nonlinear homomorphism} as  a smooth map of vector spaces $\a\co A\to B$ such that the derivative
$T\a(a)\co A\to B$
at each $a\in A$ is an algebra homomorphism in the ordinary sense. (For superalgebras, one has to consider a map $\a\co \Abold\to \Bbold$ of the associated `linear supermanifolds' $\Abold$ and $\Bbold$.) Similarly  \emph{formal  homomorphisms} are defined. These notions should lead us to a nonlinear generalization of the algebro-functional duality.

Question: does every nonlinear (or formal)  homomorphism between the algebras of smooth functions on (super)manifolds arise as the nonlinear pullback induced by some thick morphism? A positive answer would be a nonlinear counterpart of the well-known statement for ordinary homomorphism and ordinary smooth maps.

\subsection{Idea of construction} For constructing the formal categories $\EThick$ and $\OThick$  and   nonlinear pullbacks, we   use  very classical tools of mathematical physics such as canonical relations and their generating functions. To V.~I.~Arnold belongs a remark about the ``depressingly noninvariant'' nature of generating functions~\cite[\S 47]{arnold:mathmethodseng}. The positive interpretation of this fact is that    generating functions possess  {a} nontrivial transformation law  under   changes of coordinates. In our constructions, generating functions of a particular type     appear as  central  geometric objects. A \emph{thick morphism} between two supermanifolds is given by a generating function $S(x,q)$, which specifies a canonical relation between the corresponding cotangent bundles and is regarded as part of {the} structure. A generating function $S(x,q)$ is a function of positions on the source manifold and momenta on the target manifold. The action on functions, $g(y)\mapsto f(x)$, is defined in terms of this generating function as
\begin{equation*}
   \boxed{\quad f(x)=g(y) + S(x,q) - y^iq_i\,, \quad\vphantom{\int_0^1} }
\end{equation*}
where to eliminate the variables $q$ and $y$ one   uses  the coupled equations $q_i=\lder{g}{y^i}$ and $y^i= \lder{S}{q_i}$,  solved by iterations. One can show that this formula generalizes  ordinary pullback (as substitution into the argument).
As the reader will see, we have to consider   generating functions as formal power expansions in the momentum variables. This explains the adjective `microformal' in the alternative name for thick morphisms and the name \textbf{microformal geometry} for the whole  theory.\footnote{\,The root `micro-' has an established usage, e.g.,   microlocal analysis (local in the cotangent or jet directions)  and    Milnor's microbundles.  It is also used  in    `symplectic microgeometry'~\cite{cattaneo-dherin-weinstein:one,cattaneo-dherin-weinstein:two,cattaneo-dherin-weinstein:three}.}

\subsection{Plan of the paper} The exposition is organized  as follows.

In  Section~\ref{sec.categories},  we introduce the  {\emph{microformal categories}} $\EThick$ and $\OThick$,   and  develop the functorial properties of thick morphisms (the construction of pullback).

In Section~\ref{sec.adjoint}, we define the   adjoint for a nonlinear morphism of vector spaces or vector bundles  as a thick morphism of the dual bundles, with properties similar to those of the ordinary adjoints. 
The construction    uses the canonical diffeomorphism $T^*E\cong T^*E^*$ of Mackenzie--Xu~\cite{mackenzie:bialg}  and its odd analog   $\Pi T^*E\cong \Pi T^*(\Pi E^*)$ introduced in~\cite{tv:graded}. Using them, we prove  in Section~\ref{sec.algebroid}  that an $L_{\infty}$-morphism of $L_{\infty}$-algebroids induces   $L_{\infty}$-morphisms of the homotopy Lie--Poisson brackets  on the dual vector bundles and  Lie--Schouten brackets   on the antidual vector bundles. We then apply this result to the theory of higher Koszul brackets and to   triangular $L_{\infty}$-bialgebroids.

In  Section~\ref{sec.quantum}, we show that, in the bosonic case, the microformal category and nonlinear pullbacks are the classical limit (for $\hbar\to 0$) of a  \emph{quantum microformal category}, which  {is}   dual to a category whose morphisms are \new{a} particular type of Fourier  integral operators  perceived as ``quantum pullbacks''. Each such   operator is specified by a ``quantum generating function''.
Quantum pullbacks act on oscillatory wave functions, which are linear combinations of oscillatory exponentials with coefficients  in formal power series in   $\hbar$. Calculating the integrals by the stationary phase method yields   formulas for  ``classical'' thick morphisms. In hindsight, one may see this as a justification of the ``classical'' formulas.  Finally, in Section~\ref{sec.quantumhomot}, we show how the applications of thick morphisms to homotopy bracket structures can be lifted to the ``quantum'' level.

Since the quantum version of our constructions relies on the stationary phase method, we included an appendix containing the necessary statements in the form adapted for our purposes.

One clarifying remark is in order, that two different types of formal power expansions   arise here.   One expansion is present already in the classical theory (generating functions themselves, pullback, composition law). It can be compared with the ``expansion in the  coupling constant'' in field theory. Another is the expansion in $\hbar$ and gives ``quantum corrections''.

We also wish to point out   a relation between this ``microformal geometry''
and  the ``symplectic microgeometry'' of  A.~Cattaneo, B.~Dherin and A.~Weinstein. In a remarkable series of papers~\cite{cattaneo-dherin-weinstein:one,cattaneo-dherin-weinstein:two,cattaneo-dherin-weinstein:three}, see also~\cite{cattaneo-dherin-weinstein:comorphisms} and~\cite{weinstein:wehrheimwoodwardcat}, they systematically developed a `micro' analog of symplectic geometry with ``symplectic microfolds''  defined as germs of symplectic manifolds at  Lagrangian submanifolds and with germs of canonical relations as  morphisms. The  microsymplectic category  so obtained was intended to cure the problem of  partially defined multiplication in Weinstein's symplectic ``category in quotes''~\cite{weinstein:bull69,weinstein:sympl71,weinstein:symplcat82,weinstein:symplectic-geometry1981,weinstein:coisotropic88}. Our formal categories $\EThick$ and $\OThick$ are   close to this microsymplectic category. The  key difference is that in our case, (formal) canonical relations between the cotangent bundles play the role of morphisms between the bases\,---\,not between the bundles themselves\,---\,and they are introduced in order  to obtain an action on smooth functions on the bases, which is our central   concept of nonlinear pullback.\footnote{\,This action on functions on  Lagrangian submanifolds in the ambient symplectic manifolds  brings to mind the   spinor representation in its various versions; it is curious to clarify  whether this is more than a superficial resemblance.}

\subsection{Terminology and notations}
For simplicity, we often use `manifolds' for `supermanifolds' and generally suppress the prefix `super-' unless we wish to emphasize that we consider   the supercase. Also, to simplify the speech, we as a rule suppress the prefix `super-' and speak about differential forms and multivector fields,  on a supermanifold, when, strictly speaking,   pseudodifferential forms and pseudomultivector fields are discussed (i.e., by definition, arbitrary smooth functions on the bundles   $\Pi TM$ and $\Pi T^*M$, respectively).
In   notation and terminology we generally follow~\cite{tv:graded,tv:higherder,tv:napl,tv:qman,tv:qman-mack}. The parity ($\Z$-grading) of an object    is denoted by a tilde over its symbol. Tensor indices carry the parities of the corresponding coordinates. $\Pi$ stands for the parity reversion functor, on vector spaces, modules or vector bundles.
For a substantial part of our constructions,
the supergeometric context is inessential. Consideration of supermanifolds is necessary for  applications  to homotopy structures.
For applications, one may need   also  \emph{graded manifolds}, which are supermanifolds that besides the $\Z$-grading or  \emph{parity}  possess   an independent  $\ZZ$-grading or  \emph{weight}    (see ~\cite{tv:graded}, also~\cite{tv:qman,tv:qman-mack,tv:napl}). Our constructions  can be extended to the graded case without difficulty.

Throughout the paper we denote local coordinates on a manifold $M$  by $x^a$ and the canonically conjugate momenta by $p_a$. The canonical symplectic form on $T^*M$ is
\begin{equation*}
    \o=dp_adx^a= d(p_adx^a)\,.
\end{equation*}
Note that the \emph{Liouville   $1$-form}  $\t =p_adx^a$  is   defined invariantly.
When we need several manifolds, we introduce different letters for local coordinates on each of them, as well as for the corresponding conjugate momenta.


\section{`Even' and `odd' microformal categories. Main properties}\label{sec.categories}

Consider supermanifolds $M_1$ and $M_2$ with local coordinates $x^a$ and $y^i$, and the corresponding conjugate momenta  $p_a$ and $q_i$ (coordinates on the cotangent spaces). Let $T^*M_2\times (-T^*M_1)$ denote the product $T^*M_2\times T^*M_1$ equipped with the symplectic form\footnote{\,We have changed notations in comparison with~\cite{tv:nonlinearpullback}, where $T^*M_1\times(-T^*M_2)$ was used. The order $T^*M_2\times (-T^*M_1)$ is more traditional in symplectic geometry. Note that it is also convenient to regard the graphs of maps $f\co X\to Y$ as subspaces of $Y\times X$, not $X\times Y$.}
\begin{equation*}
    \o=\o_2-\o_1=d(q_idy^i-p_adx^a)\,.
\end{equation*}

\begin{definition}\label{def.thick}
 A \emph{thick morphism} (or \emph{microformal morphism}) $\Phi\co M_1\tto M_2$ is defined as a formal canonical relation (which we denote by the same letter) $\Phi\subset T^*M_2\times (-T^*M_1)$\,, together with an even  function $S=S(x,q)$  defined in each local coordinate system  and depending as arguments on position variables on the source manifold and  momentum variables on the target manifold, such that
 \begin{equation}\label{eq.relphi}
    \Phi=\left\{\vphantom{\der{S}{x^a}(x,q)}\,(y^i,q_i\,;\, x^a,p_a)\ \right|\left. \ y^i=(-1)^{\itt}\,\der{S}{q_i}(x,q)\,, \ p_a=\der{S}{x^a}(x,q)\right\}\,.
 \end{equation}
We  call the function $S=S(x,q)$,  the \emph{generating function} of a thick morphism. It is considered part of \new{the} structure.
\end{definition}

We shall elaborate this definition below, but first give an example.

\begin{example} Consider a smooth map $\f\co M_1\to M_2$. In local coordinates, it is given by $y^i=\f^i(x)$. Set
\begin{equation}\label{eq.Slinear}
    S(x,q)=\f^i(x)q_i\,.
\end{equation}
This function gives a canonical relation $R_{\f}\subset T^*M_2\times (-T^*M_1)$ specified by the equations (as one immediately sees)
\begin{equation*}
    y^i=\f^i(x)\,, \ p_a=\der{\f^i}{x^a}(x)q_i\,.
\end{equation*}
The relation $R_{\f}$ is the canonical lifting of a map $\f$ to the cotangent bundles. In a coordinate-free language,
\begin{equation*}
    R_{\f}=\gr(\bar \f)\circ \left(\gr(T^*\f)\right)^{\op}\,,
\end{equation*}
where $\bar\f\co \f^*(T^*M_2)\to T^*M_2$ is the vector bundle morphism which is identity   on the fibers and   covers a map of the bases   $\f\co M_1\to M_2$,      and $T^*\f\co \f^*(T^*M_2)\to T^*M_1$ is the dual to the tangent map $T\f\co TM_1\to \f^*(TM_2)$.
\end{example}

It follows that we can identify  ordinary smooth maps $\f\co M_1\to M_2$ with a subclass of  thick morphisms $M_1\tto M_2$  specified by generating functions $S(x,q)$ of the form~\eqref{eq.Slinear}, i.e., linear in  momenta.

Consider now the general case.  Recall that  a canonical relation  or   correspondence  $\Phi$ between  symplectic manifolds $N_1$ and $N_2$ (in our case these are $T^*M_1$ and $T^*M_2$)  is a Lagrangian submanifold in the product $N_2\times (-N_1)$ taken with the form $\o=\o_2-\o_1$. Such relations are customarily  perceived as   partial multi-valued mappings $N_1\dashrightarrow N_2$ (direction of the arrow being a matter of convention)  that generalize  symplectomorphisms. However, this is not the intuition that we shall follow. For us    this relation or correspondence $\Phi$ is an analog of a map between the manifolds $M_1$ and $M_2$ themselves (and not between  their cotangent bundles). For our purposes we consider not arbitrary canonical relations, but only of a particular kind, those that are specified by generating functions of the type $S(x,q)$. 
(In particular, unlike for relations in general, the direction    from $M_1$ to $M_2$  in our constructions is unambiguous and not a matter of convention.)

To understand the role of the generating function $S$ in Definition~\ref{def.thick}, recall that for an arbitrary Lagrangian submanifold $\Phi\subset T^*M_2\times (-T^*M_1)$ the $1$-form $q_idy^i-p_adx^a$ is closed, hence locally exact, i.e.,  there is  a function $F$  on $\Phi$ defined independently of a choice of coordinates (but possibly  only locally and up to a constant), such that
\begin{equation}\label{eq.dF}
    q_idy^i-p_adx^a =dF\,.
\end{equation}
In Definition~\ref{def.thick} it is  assumed that the variables $x^a$ and $q_i$
yield a system of local coordinates on the submanifold $\Phi$.  \new{(This follows the case of an ordinary map.)}  The equations specifying $\Phi$  mean that
$p_adx^a+(-1)^{\itt}y^idq_i=dS$,
which is equivalent to
\begin{equation}\label{eq.dS}
    q_idy^i-p_adx^a =d(y^iq_i-S)\,.
\end{equation}
The l.h.s. of~\eqref{eq.dS} is invariant, but the explicit appearance of the variables $y^i$ and $q_i$ in the r.h.s.    makes the 
function $S(x,q)$ a coordinate-dependent object (unlike $F$ in~\eqref{eq.dF}). We shall give below the precise transformation law for $S$. The functions $S$   and $F$ are related by a    Legendre transform type formula,
\begin{equation}\label{eq.sandf}
    F=y^iq_i-S\,.
\end{equation}
(It is an actual  Legendre transform  if $F$ can be regarded as a function of independent variables $x^a,y^i$, which may   not necessarily hold in general.)
The relation $\Phi$     defines  only the differentials $dF$ or $dS$. We  assume that constants of integration are chosen, so that we can speak unambiguously about  the functions $F$ or $S$, and  that   $F$  can be defined globally.

What is a  coordinate-free characterization of the considered type of Lagrangian submanifolds? The condition that $x^a$ be independent on $\Phi$  is equivalent to   the submanifold $\Phi$ project on $M_1$ without degeneration (with full rank). In contrast with that, the second  condition that $q_i$ be independent on $\Phi$ is equivalent to  $\Phi$ ``project   without degeneration on the fibers of $T^*M_2$'', but this seems not have a well-defined meaning without a choice of a local trivialization. Consider, however, the 
differentials $dq_i$. We have $q_i=\lder{y^{i'}}{y^i}\,q_{i'}$, so  obtain
\begin{equation*}
    dq_i=d\!\left(\der{y^{i'}}{y^i}\right)q_{i'} +(1)^{\itt+\itt'}\,\der{y^{i'}}{y^i}\;dq_{i'}\,.
\end{equation*}
We see that when $q_{i'}$ are small (i.e., we are near the zero section of $T^*M_2$), the linear independence of  $dq_i$ on $\Phi$  implies the linear independence of  $dq_{i'}$, and vice versa. Therefore we   conclude that the condition that the variables $q_i$ be independent on $\Phi$ (or   ``$\Phi$ project    without degeneration on the fibers of $T^*M_2$'') has invariant meaning on a small neighborhood of the zero section of $T^*M_2$.  In particular, it makes sense on the  formal neighborhood of $M_2$ in $T^*M_2$. Therefore  we define  $\Phi$ as   a \emph{formal} canonical relation, i.e., a Lagrangian submanifold of the formal neighborhood\footnote{\,Replacing a formal submanifold by a germ  would give  a `symplectic micromorphism' between   `symplectic microfolds' represented by the pairs $(T^*M_1,M_1)$ and   $(T^*M_2,M_2)$, a notion introduced by Cattaneo--Dherin--Weinstein.    Note that our  thick morphisms   are morphisms between $M_1$ and $M_2$, while    symplectic micromorphisms  are  morphisms   between objects of double  dimensions.}  of $M_2\times T^*M_1$ in $T^*M_2\times (-T^*M_1)$.

Hence we consider  the generating function $S(x,q)$ of a thick morphism $\Phi\co M_1\tto M_2$ as a formal power series
\begin{equation}\label{eq.sexpand}
    S(x,q)=S^0(x)+ S^i(x)q_i + \frac{1}{2}\,S^{ij}(x)q_jq_i + \frac{1}{3!}\,S^{ijk}(x)q_kq_jq_i +\ldots
\end{equation}
in the momentum variables $q_i$\,.
In the sequel we   frequently suppress the adjective `formal' for various objects that we consider (functions, submanifolds, etc.). As we shall see, it makes sense to group the terms in this expansion as
\begin{equation}\label{eq.sexpandgroup}
    S(x,q)=S^0(x)+ S^i(x)q_i + S^{+}(x,q)\,,
\end{equation}
where $S^{+}(x,q)$ contains all terms of order $2$ and higher in $q_i$.

To conclude elaborating our definition, we state the following transformation law for their generating functions $S$. For logical simplicity we may regard  it as    part of  the   definition, but  it can be deduced from  equations~\eqref{eq.dF},\eqref{eq.dS},\eqref{eq.sandf} together with the invariance condition for a submanifold $\Phi$.

\begin{transflaw}[for generating functions] A  \emph{generating function}  $S$ of a thick morphism $\Phi\co M_1\tto M_2$  as a geometric object on $M_1\times M_2$ transforms by
\begin{equation}\label{eq.news}
    S'(x',q')=S(x,q) - y^iq_i +y^{i'}q_{i'}\,,
\end{equation}
under an invertible change of local coordinates $x^a=x^a(x')$, $y^{i}=y^{i}(y')$. Here $S(x,q)$ is the expression for $S$ in `old' coordinates and $S'(x',q')$ is the expression for $S$ in `new' coordinates. The variables $x^a$ and $y^{i'}$ in the r.h.s. of~\eqref{eq.news} are given simply by the  substitutions: $x^a=x^a(x')$ and  $y^{i'}=y^{i'}(y)$ (where as usual $y^{i'}=y^{i'}(y)$ is the inverse change of coordinates), while
 $q_i$   and    $y^i$   are determined from the coupled
equations
\begin{equation} \label{eq.newsadd}
   q_i=\der{y^{i'}}{y^i}(y)\,q_{i'}\,, \quad y^i=(-1)^{\itt}\der{S}{q_i}(x,q)\,.
\end{equation}
\end{transflaw}

\begin{proposition} The transformation law~\eqref{eq.news} satisfies the cocycle condition (hence, in particular, the set of generating functions $S$ is non-empty). A generating function $S$ with local representations $S(x,q)$ and the transformation law~\eqref{eq.news} specifies a well-defined formal canonical relation $\Phi\subset T^*M_2\times (-T^*M_1)$.
\end{proposition}
\begin{proof} The cocycle condition  immediately follows because the transformation law~\eqref{eq.news} has  a ``coboundary''  form. Equation~\eqref{eq.news} also means that functions $y^iq_i -S(x,q)$ glue into one global function.
To check the second statement, we need to show that if \eqref{eq.dS} holds for $S$, $x^a$, $p_a$, $y^i$, $q_i$, and $S'$ is related with $S$ by the given transformation law, then the same   relation
\begin{equation}\label{eq.dSdash}
    q_{i'}dy^{i'}-p_{a'}dx^{a'} =d(y^{i'}q_{i'}-S')\,,
\end{equation}
holds for the `new' variables $S'$, $x^{a'}$, $p_{a'}$, $y^{i'}$, $q_{i'}$ (assuming the standard transformation laws for the positions and momenta)\,. But the l.h.s. of~\eqref{eq.dSdash} equals $q_idy^i-p_adx^a$ by the invariance of the Liouville forms and $y^{i'}q_{i'}-S'$ in the r.h.s. equals $y^iq_i-S$ by~\eqref{eq.news}. Hence, \eqref{eq.dS} and~\eqref{eq.dSdash} are equivalent.
\end{proof}

\begin{example} \label{ex.s0phi}
Consider a generating function $S$ that in one coordinate system has the form
\begin{equation}\label{eq.s0andphi}
    S(x,q)=S^0(x)+\f^i(x)q_i
\end{equation}
(we write $\f^i$ instead of $S^i$ for convenience, as will become clear shortly). Explore the action of the transformation law on $S$. We have
\begin{equation*}
    S'(x',q')=S(x,q)-y^iq_i+y^{i'}q_{i'}= S^0(x)+\f^i(x)q_i -y^iq_i+y^{i'}q_{i'}\,,
\end{equation*}
where we have to substitute $x=x(x')$, $y'=y'(y)$, and for $y$ and $q'$ we need to solve the equations~\eqref{eq.newsadd}. But in our case,
they decouple and for $y$ simply give
\begin{equation*}
    y^i=(-1)^{\itt}\der{S}{q_i}(x,q)=\f^i(x)\,.
\end{equation*}
Hence the terms $\f^i(x)q_i -y^iq_i$ in $S'$ cancel and we obtain (taking into account the substitutions $y'=y'(y)$, $y=\f(x)$, and $x=x(x')$)
\begin{equation*}
    S'(x',q')= S^0(x)+y^{i'}q_{i'}=S^0(x(x'))+y^{i'}(\f(x(x')))q_{i'}\,.
\end{equation*}
In other words, in new coordinates $S$ has the same form
\begin{equation*}
    S'(x',q')= S^{'0}(x')+\f^{i'}q_{i'}\,,
\end{equation*}
where
\begin{equation*}
    S^{'0}(x')= S^0(x(x'))
\end{equation*}
and
\begin{equation*}
    \f^{i'}= y^{i'}(\f(x(x')))\,.
\end{equation*}
These are precisely the transformation laws for coordinate representations of a scalar function on $M_1$ and   a map $\f\co M_1\to M_2$\,.
\end{example}

We conclude that thick morphisms $\Phi\co M_1\tto M_2$ with generating functions $S$ of the form~\eqref{eq.s0andphi}, of degree $\leq 1$ in momenta, invariantly correspond to pairs $(\f,S^0)$ where $\f\co M_1\to M_2$ is a smooth map and $S^0\in \fun(M_1)$ is an (even) smooth function on the source manifold. (We shall see later that such pairs are morphisms in a semidirect product category.)

\begin{example} \label{ex.s0phisplus}
Consider now the general case where the generating function of a thick morphism $\Phi\co M_1\tto M_2$ has the form~\eqref{eq.sexpandgroup}. We rewrite it as
\begin{equation}\label{eq.sexpandgroup2}
    S(x,q)=S^0(x)+\f^i(x)q_i +S^{+}(x,q)\,,
\end{equation}
having in mind the previous example. Let us analyze how the particular terms in~\eqref{eq.sexpandgroup2} transform. The transformation law gives
\begin{equation*}
    S'(x',q')=S(x,q)-y^iq_i+y^{i'}q_{i'}= S^0(x)+\f^i(x)q_i +S^{+}(x,q)  -y^iq_i+y^{i'}q_{i'}\,,
\end{equation*}
where as before we have to substitute $x=x(x')$, $y'=y'(y)$, and   $y$ and $q$ are obtained by  solving  equations~\eqref{eq.newsadd}. But now the equation for determining $y$ takes the form
\begin{equation*}
    y^i=\f^i(x)+ (-1)^{\itt}\der{S^{+}}{q_i}\left(x,\der{y'}{y}(y)q'\right)\,;
\end{equation*}
note that the second term is of order $\geq 1$ in $q'$. This gives a unique solution as a power series in $q'$, of the form
\begin{equation*}
    y^i=\f^i(x)+ y^{+i}(x,q')
\end{equation*}
(the second term of order $\geq 1$ in $q'$). From here
\begin{equation*}
    q_{i}=\der{y^{i'}}{y^i}\left(\f(x)+y^{+}(x,q')\right)q_{i'}\,,
\end{equation*}
and for $S'$ we arrive at
\begin{multline*}
    S'(x',q')= S^0(x)+\left(\f^i(x)-y^i\right)q_i +S^{+}(x,q)  +y^{i'}q_{i'}
    =\\
    S^0(x)- y^{+i}(x,q')\,q_i +S^{+}(x,q) + y^{i'}\bigl(\f(x)+y^{+}(x,q')\bigr)q_{i'}=\\
    S^0(x)+ y^{i'}\bigl(\f(x)+y^{+}(x,q')\bigr)q_{i'}
    - y^{+i}(x,q')\,\der{y^{i'}}{y^i}\left(\f(x)+y^{+}(x,q')\right)q_{i'} \\
    +S^{+}\Bigl(x,\der{y'}{y}\left(\f(x)+y^{+}(x,q')\right)q'\Bigr)
    \,,
\end{multline*}
where we need to  substitute finally $x=x(x')$. In particular we obtain
\begin{equation*}
    S'(x',q')\equiv S^0(x(x'))+ y^{i'}\bigl(\f\bigl(x(x')\bigr)\bigr)q_{i'} \mod \langle q'\rangle^2\,.
\end{equation*}
Hence
\begin{equation*}
    S'(x',q')=S^{'0}(x')+\f^{i'}(x')q_{i'} +S^{'+}(x',q')\,,
\end{equation*}
 where
\begin{equation*}
    S^{'0}(x')= S^0(x(x'))
\end{equation*}
and
\begin{equation*}
    \f^{i'}= y^{i'}(\f(x(x')))\,.
\end{equation*}
That means that the first two terms in the expansion~\eqref{eq.sexpandgroup} or \eqref{eq.sexpandgroup2} represent, respectively, a scalar function on $M_1$ and   a map $\f\co M_1\to M_2$. At the same time, the transformation law for  the term
$S^{+}$ includes higher derivatives of  changes  of coordinates on $M_2$ calculated at the points $\f(x)$\,.
\end{example}

From Examples~\ref{ex.s0phi} and \ref{ex.s0phisplus},
we see that pairs $(\f,S^0)$ correspond to thick morphisms $M_1\tto M_2$ of a special type and, conversely, an arbitrary thick morphism  $\Phi\co M_1\tto M_2$ canonically defines such a pair. So we have an ``inclusion--retraction'' setting. We shall come back to that.

Out next task is to define action of thick morphisms on functions.

Consider the algebras of smooth functions $\fun(M)$. For each supermanifold $M$, the algebra   $\fun(M)$ is a commutative $\Z$-graded  algebra. We shall regard smooth functions of particular parity on $M$ as points of an  infinite-dimensional supermanifold. (The word ``smooth'' will be often omitted in the sequel.) We have the    \emph{supermanifold of all even functions} on $M$, which we denote  $\funn(M)$, and the   \emph{supermanifold of all odd functions} on $M$, which we denote $\pfunn(M)$. We use boldface to distinguish vector supermanifolds from the corresponding to them $\Z$-graded linear spaces. (A physicist would say that the points of $\funn(M)$ are   `bosonic fields',  and the points of $\pfunn(M)$ are  `fermionic fields' on $M$.)

\begin{definition}[\cite{tv:nonlinearpullback}]
Let $\Phi\co M_1\tto M_2$ be a thick morphism with a generating function $S$. The \emph{pullback} $\Phi^*$ is  a formal mapping of functional supermanifolds of even functions, $g\mapsto \Phi^*[g]$,
 \begin{equation}\label{eq.phistar1}
    \Phi^*\co \funn(M_2) \to \funn(M_1)\,,
 \end{equation}
defined by
\begin{equation}\label{eq.phistar2}
    \Phi^*[g] (x)= g(y) + S(x,q) - y^iq_i\,,
\end{equation}
where $q_i$ and $y^i$ are determined from the equations
\begin{equation}\label{eq.phistarq}
    q_i=\der{g}{y^i}\,(y)
\end{equation}
and
\begin{equation}\label{eq.phistary}
    y^i=(-1)^{\itt}\,\der{S}{q_i}(x,q)\,.
\end{equation}
Here $g\in \funn(M_2)$ is an even function on $M_2$ and $\Phi^*[g]$ is its image in $\funn(M_1)$.
\end{definition}

\begin{remark} We showed in~\cite{tv:nonlinearpullback} that the pullback $\Phi^*$ does not depend on a choice of coordinates. This is guaranteed by the transformation law of the generating function $S$.
\end{remark}

\begin{example} Consider a thick morphism $\Phi\co M_1\tto M_2$ defined by a pair $(\f,S^0)$. We have
\begin{equation}
    S(x,q)=S^0(x)+\f^i(x)q_i\,.
\end{equation}
From~\eqref{eq.phistary}, we obtain $y^i=\f^i(x)$, so
\begin{equation*}
    \Phi^*[g](x)= g(y) + S(x,q) - y^iq_i= g(y) + S^0(x)+\f^i(x)q_i - y^iq_i = g(\f(x)) +S^0(x)\,.
\end{equation*}
Hence $\Phi^*$ in this case is an affine transformation,
\begin{equation}
    \Phi^*[g]=S^0+ \f^*(g)\,,
\end{equation}
the combination of the ordinary pullback by a map $\f\co M_1\to M_2$ and the shift by a function $S^0\in \funn(M_1)$. (In particular, formula~\eqref{eq.phistar2} gives the usual pullback when a thick morphism is an  ordinary smooth map.)
\end{example}

Let us see how the construction of $\Phi^*$ works in general.

Substituting~\eqref{eq.phistarq} into~\eqref{eq.phistary}  gives the equation for $y^i$,
\begin{equation}\label{eq.y}
    y^i=(-1)^{\itt}\,\der{S}{q_i}\Bigl(x,\der{g}{y}\,(y)\!\Bigr)\,,
\end{equation}
which can be solved by iterations.
If we use~\eqref{eq.sexpandgroup2}, the equation takes the form
\begin{equation}\label{eq.yphi}
    y^i=\f^i(x)+ (-1)^{\itt}\,\der{S^{+}}{q_i}\Bigl(x,\der{g}{y}\,(y)\!\Bigr)\,,
\end{equation}
where the second term is of order $\geq 1$ in $\lder{g}{y}$. There is a unique solution for $y$ as a ``functional'' power series in $g$. More precisely, a formal power series in the first and higher derivatives of $g$ evaluated at   $y=\f(x)$ and starting from $y=\f(x)$ as the zero-order term. This gives a  ``perturbed''  map $\f_g\co M_1\to M_2$ depending on $g\in\funn(M_2)$\,, as a series
\begin{equation}\label{eq.phig}
    \f_g=\f+\f_{(1)}+\f_{(2)}+\ldots \,,
\end{equation}
where $\f\co M_1\to M_2$ is    defined  by the thick morphism $\Phi$ and does not depend on $g$, while
the next terms $\f_{(k)}$ give ``higher corrections'' to $\f$ (linear, quadratic, etc., in the function $g$).
Using   $\f_g$, the pullback $\Phi^*[g]$ can be expressed as
\begin{equation}\label{eq.phistarexpl}
    \Phi^*[g](x)= g\bigl(\f_g(x)\bigr) + S\Bigl(x,\der{g}{y}\bigl(\f_g(x)\bigr)\Bigr) - \f_g^i(x)\,\der{g}{y^i}\bigl(\f_g(x)\bigr)\,,
\end{equation}
which   demonstrates the nonlinear dependence on $g$. In terms of \eqref{eq.sexpandgroup2}, we obtain after simplification:
\begin{multline} \label{eq.phig2}
    \Phi^*[g](x)= S^0(x)+ g\bigl(\f(x)+\f_{(1)}(x)+\ldots \bigr)-
    \\
     \bigl(\f_{(1)}^i(x) +\ldots \bigr)\,\der{g}{y^i}\bigl(\f(x)+\f_{(1)}(x)+\ldots \bigr) 
    +
    S^{+}\Bigl(x,\der{g}{y}\bigl(\f(x)+\f_{(1)}(x)+\ldots \bigr)\Bigr)\,.
\end{multline}

\begin{example} Calculate $\Phi^*[g]$ to the second order in $g$. From~\eqref{eq.phig2}, we immediately see that the terms of order $\leq 1$ are precisely
\begin{equation*}
    S^0(x)+ g\bigl(\f(x)\bigr)\,.
\end{equation*}
For the quadratic correction, there are inputs from the three last summands in~\eqref{eq.phig2}, but two of them cancel\,:
\begin{equation*}
    \f_{(1)}^i(x)\der{g}{y^i}\bigl(\f(x)\bigr) - \f_{(1)}^i(x)\der{g}{y^i}\bigl(\f(x)\bigr) +
    S^{+}_{(2)}\Bigl(x,\der{g}{y}\bigl(\f(x)\bigr)\Bigr)= S^{+}_{(2)}\Bigl(x,\der{g}{y}\bigl(\f(x)\bigr)\Bigr)\,.
\end{equation*}
Here $S^{+}_{(2)}(x,q)=\frac{1}{2}\,S^{ij}(x)q_jq_i$ is the quadratic term in the expansion of $S$. Altogether,
\begin{equation}
    \Phi^*[g](x)= S^0(x)+ g\bigl(\f(x)\bigr) + \frac{1}{2}\,S^{ij}(x)\,\p_i g\bigl(\f(x)\bigr)\p_j g\bigl(\f(x)\bigr)
    + \ldots \,.
\end{equation}
\end{example}

This is the general pattern: the pullback $\Phi^*\co \funn(M_2) \to \funn(M_1)$ with respect to a thick morphism is a formal nonlinear differential operator, so that  the   terms of order $k$ in $g$ of the expansion of $\Phi^*[g]$ are   homogeneous polynomials of degree $k$ in the derivatives of $g$ of orders $\leq k-1$ evaluated at $y=\f(x)$\,, with the zeroth and   first order terms being the combination of the shift   and  ordinary pullback: $S^0+\f^*(g)$\,. We see again the different roles of the three summands  in the expansion~\eqref{eq.sexpandgroup}, \eqref{eq.sexpandgroup2}.

\begin{remark} Pullbacks with respect to thick morphisms can be applied to functions  defined   on an open domain $U\subset M_2$. The image  will be in $\funn(\f^{-1}(U))$, where $\f\co M_1\to M_2$ is the underlying  ordinary map.  
\end{remark}

\begin{example} \label{ex.testfunction}
If we apply $\Phi^*$ to the function $g=y^ic_i$, where $y^i$ are local coordinates on $M_2$ and $c_i$ are some auxiliary variables, then we obtain $q_i=c_i$ from ~\eqref{eq.phistarq}  and
\begin{equation*}
    \Phi^*[y^ic_i]= y^ic_i- y^iq_i+S(x,q)=S(x,c)\,.
\end{equation*}
In this way we recover the generating function $S=S(x,q)$.
\end{example}

A thick morphism $\Phi$ is therefore determined by  the   action of  $\F^*$ on  linear combinations of   coordinate functions.
Hence, although the pullback $\Phi^*$ is a nonlinear mapping, 
it still respects some algebraic properties  such as the role of local coordinates as   `free generators'.\footnote{The algebra of smooth functions on a coordinate (super)domain   is not of course a free algebra in the standard algebraic sense with respect to arbitrary homomorphisms (which would be the polynomial algebra), but it behaves as a free algebra with respect to the homomorphisms induced by smooth maps, which are defined by the images of the coordinate functions not subject to any restrictions.}

In~\cite{tv:nonlinearpullback}, we proved the following statement that points at another aspect of algebraic properties of pullbacks $\Phi^*$.

\begin{theorem}[\cite{tv:nonlinearpullback}] \label{thm.tphi}
For every function $g\in\funn(M_2)$, the tangent map
\begin{equation*}
    T\Phi^*[g]\co \fun(M_2)\to \fun(M_1)
\end{equation*}
for the pullback $\Phi^*\co \funn(M_2)\to \funn(M_1)$ by a thick morphism $\Phi\co M_1\tto M_2$ is the  ordinary  pullback $\f_g^*$ by the map
\begin{equation*}
    \f_g\co M_1\to M_2
\end{equation*}
that corresponds to the function $g$\,. \qed
\end{theorem}

(Note that    tangent spaces to $\funn(M)$  can be identified with $\fun(M)$\,.)

We see that though the pullback  with respect to a thick morphism  as a mapping  between the vector supermanifolds corresponding to the algebras of smooth functions  is in general a nonlinear (and indeed formal) mapping,  and as such cannot be  an algebra homomorphism  in the usual sense,
it possesses the remarkable property that its derivative  (= tangent map or linearization) at each point is an algebra homomorphism. It is tempting to give the following definition.

\begin{definition} \label{def.nonlinhom}
Let $A$, $B$ be (super)algebras and   $\Abold$, $\Bbold$ denote the corresponding vector supermanifolds. A  map (a formal map) $\a\co \Abold \to \Bbold$ is a  \emph{nonlinear algebra homomorphism} (resp.,  a  \emph{formal nonlinear algebra homomorphism}) if its derivative $T\a(\abold)\co A\to B$ is an algebra homomorphism for every $\abold\in \Abold$.
\end{definition}

(The distinction between  $A$ and $\Abold$, $B$ and $\Bbold$, is important only in the super case.)

Pullbacks with respect to thick morphisms are formal nonlinear algebra homomorphisms. (In the abstract case, it is unclear whether   formal or non-formal versions of the notion is more important.) Following the known statement for ordinary algebra homomorphisms, we are tempted to suggest a conjecture:
\begin{conjecture} \label{conj.nonlinhom}
For smooth (super)manifolds $M_1$ and $M_2$ \emph{(with the usual assumptions leading to paracompactness)}, every formal nonlinear   algebra homomorphism
\begin{equation*}
    \a\co \funn(M_2)\to \funn(M_1)
\end{equation*}
is the pullback, $\a=\Phi^*$,   with respect to some  thick morphism
\begin{equation*}
    \Phi\co M_1\tto M_2\,.
\end{equation*}
\end{conjecture}
(So far we do not know if this is true or not.)

Now we wish to establish   categorical properties of thick morphisms.

Consider thick morphisms $\Phi_{21}\co M_1\tto M_2$ and $\Phi_{31}\co M_2\tto M_3$ with   generating functions $S_{21}=S_{21}(x,q)$ and $S_{32}=S_{32}(y,r)$, respectively. Here $z^{\mu}$ are local coordinates on $M_3$ and by $r_{\mu}$ we denoted the corresponding conjugate momenta.

\begin{theorem} \label{thm.classcompos}
The composition $\Phi_{32}\circ \Phi_{21}$  is well-defined as a thick morphism
\begin{equation*}
   \Phi_{31}\co M_1\tto M_3
\end{equation*}
with the  generating function $S_{31}=S_{31}(x,r)$, where
\begin{equation}\label{eq.S31}
    S_{31}(x,r)= S_{32}(y,r) + S_{21}(x,q) - y^iq_i
\end{equation}
and $y^i$ and $q_i$ are  expressed through $(x^a, r_{\mu})$  from the system
\begin{align}\label{eq.q31}
    q_i&= \der{S_{32}}{y^i}\,(y,r)
\intertext{and}
 \label{eq.y31}
    y^i&=(-1)^{\itt}\,\der{S_{21}}{q_i}\,(x,q)\,,
\end{align}
which has a unique solution   as a  power series in $r_{\mu}$ and   a functional power series in $S_{32}$.
\end{theorem}
\begin{proof} To find the composition of $\Phi_{32}$ and $\Phi_{21}$ as relations,  $\Phi_{32}\subset T^*M_3\times  T^*M_2$ and $\Phi_{21}\subset T^*M_2\times  T^*M_1$, we need to consider all pairs $(z,r\,;\,x,p)\in T^*M_3\times  T^*M_1$ for which there exist $(y,q)\in T^*M_2$ such that $(z,r\,;\,y,q)\in \Phi_{32}\subset T^*M_3\times  T^*M_2$ and $(y,q\,;\, x,p)\in \Phi_{21}\subset T^*M_2\times  T^*M_1$.  By the definition of $\Phi_{21}$, we should have
\begin{equation*}
    y^i=(-1)^{\itt}\,\der{S_{21}}{q_i}\,(x,q)\,,
\end{equation*}
$x^a$, $q_i$ are free variables,
and by the definition of $\Phi_{32}$, we should have
\begin{equation*}
    q_i= \der{S_{32}}{y^i}\,(y,r)\,,
\end{equation*}
now $y^i$, $r_{\mu}$ are free variables. Therefore we arrive at the system~\eqref{eq.q31}, \eqref{eq.y31} where $y^i$ and $q_i$ are  to be determined and  the variables $x^a, r_{\mu}$ are free. Substituting~\eqref{eq.q31} into~\eqref{eq.y31}, we obtain for determining $y$ the equation
\begin{equation*}
    y^i=(-1)^{\itt}\,\der{S_{21}}{q_i}\,\Bigl(x,\der{S_{32}}{y}\,(y,r)\Bigr)\,,
\end{equation*}
which has a unique solution $y^i=y^i(x,r)$ by iterations, similarly to the construction of pullback. Here the `parameter of smallness' is $S_{32}$, more precisely its derivative in $y^i$ in the lowest order in $r_{\mu}$.  The solution for $y^i$ can be substituted back to~\eqref{eq.q31} to obtain an expression $q_i=q_i(x,r)$. It remains to show that this composition of relations is indeed specified by  the generating function given by~\eqref{eq.S31}. We have
\begin{equation*}
    q_idy^i-p_adx^a =d(y^iq_i-S_{21})
\end{equation*}
and
\begin{equation*}
    r_{\mu}dz^{\mu}-q_idy^i =d(z^{\mu}r_{\mu} -S_{32})\,.
\end{equation*}
We obtain
\begin{equation*}
    r_{\mu}dz^{\mu}-p_adx^a =d(z^{\mu}r_{\mu}-S_{32}+y^iq_i-S_{21})\,.
\end{equation*}
Therefore, $S_{31}=S_{32}-y^iq_i+S_{21}$,  as claimed.
\end{proof}

\begin{theorem} The composition of thick morphisms is associative.
\end{theorem}
\begin{proof}
 Consider the diagram
 \begin{equation*}
  M_1\ttto{\Phi_{21}} M_2\ttto{\Phi_{32}} M_3 \ttto{\Phi_{43}} M_4\,.
 \end{equation*}
Let $\Phi_{42}=\Phi_{43}\circ \Phi_{32}$ and $\Phi_{31}=\Phi_{32}\circ \Phi_{21}$. We need to check that $\Phi_{43}\circ \Phi_{31}= \Phi_{42}\circ \Phi_{21}$\,. Consider the generating functions. 
For the \LHS, we obtain $S_{43}+S_{31}-z^{\mu}r_{\mu}=S_{43}+S_{32}+S_{21}-y^iq_i-z^{\mu}r_{\mu}$.  For the \RHS, we obtain $S_{42}+S_{21}-y^iq_i=S_{43}+S_{32}-z^{\mu}r_{\mu}+S_{21}-y^iq_i$, and the associativity follows.
\end{proof}

\begin{remark}
Since there is the identity thick morphism for each supermanifold $M$, given by the generating function $S=x^aq_a$, we conclude  that thick morphisms form a \emph{formal category}, which we denote $\EThick$ (with the same set of objects   as  the usual category of supermanifolds). `Formality' of the category means that the composition law is given by a    power series. Formality enters   our constructions in two related but different ways: as \emph{micro}formality, i.e.,     power expansions in the cotangent directions, and as formal ``funactional''  expansions in the formulas for pullback  and for the  generating function  of composition.
\end{remark}

\begin{example} Let us compute the composition of thick morphisms in the lowest order. Suppose $\Phi_{21}$ and $\Phi_{32}$ are given by generating functions
\begin{align}
 S_{21}(x,q)&=f_{21}(x) +\f_{21}^i(x)q_i+\ldots\,, \\
 S_{32}(y,r)&=f_{32}(y) +\f_{32}^{\mu}(y)r_{\mu}+\ldots\,.
\end{align}
We need to determine the generating function for the composition $\Phi_{32}\circ \Phi_{21}$,
\begin{equation}
 S_{31}(x,r)=f_{31}(x)+\f_{31}^{\mu}(x)r_{\mu}+\ldots \
\end{equation}
(Here   dots stand for the terms of higher order in momenta.) We have, in the lowest order,
\begin{multline*}
 S_{31}(x,r)=S_{32}(y,r)+ S_{21}(x,q)-y^iq_i= f_{32}(y) +\f_{32}^{\mu}(y)r_{\mu}+ f_{21}(x) +\f_{21}^i(x)q_i -y^iq_i +\ldots= \\
 f_{32}(y) +\f_{32}^{\mu}(y)r_{\mu}+ f_{21}(x) +\ldots = f_{32}(\f_{21}(x)) +\f_{32}^{\mu}(\f_{21}(x))r_{\mu}+ f_{21}(x) +\ldots\,.
\end{multline*}
Here we are calculating modulo   $J^2$ where the ideal $J$ is generated by the momenta and the zero-order terms such as $f_{21}$. Note that $y^i$ have to be determined only modulo $J$, so from~\eqref{eq.y31}, $y^i=\f_{21}^i(x) \mod J$, and the terms $\f_{21}^i(x)q_i$ and $y^iq_i$ mutually cancel. Therefore we see that
\begin{align}
 f_{31}&=\f_{21}^*(f_{32}) + f_{21} \,,\\
 \f_{31}&=\f_{32}\circ \f_{21}\,.
\end{align}
That means that, in the lowest order, we obtain the composition in the semidirect product category $\SMan\rtimes \funn$. The objects in this category are supermanifolds and the morphisms are pairs $(\f_{21},f_{21})$, where $\f_{21}\co M_1\to M_2$ is a supermanifold map and $f_{21}\in\funn(M_1)$ is an even function on the source supermanifold,  with the composition of pairs $(\f_{32},f_{32})\circ (\f_{21},f_{21})=(\f_{32}\circ \f_{21}, \f_{21}^*f_{32}+f_{21})$.
\end{example}

\begin{remark} The  category $\SMan\rtimes \funn$ is a closed subspace in the formal category $\EThick$, and the whole  $\EThick$ is its formal neighborhood. Our calculations show that there are inclusion and retraction functors
\begin{equation*}
   \SMan\rtimes \funn \rightleftarrows \EThick\,.
\end{equation*}
\end{remark}

\begin{theorem}
 For     pullbacks defined by thick morphisms the identity
 \begin{equation}
  (\Phi_{32}\circ \Phi_{21})^*=\Phi_{21}^*\circ \Phi_{32}^*
 \end{equation}
holds.
\end{theorem}
\begin{proof}
Consider $f_3\in\funn(M_3)$. Then for $\Phi_{32}^*[f_3]$ we have
 \begin{equation*}
    \Phi_{32}^*[f_3]=f_3+S_{32}-z^{\mu}r_{\mu}
 \end{equation*}
 and for $(\Phi_{21}^*\circ \Phi_{32}^*)[f_3]$ we obtain
 \begin{equation*}
    (\Phi_{21}^*\circ \Phi_{32}^*)[f_3]=\Phi_{21}^*[\Phi_{32}^*[f_3]]=f_3+S_{32}-z^{\mu}r_{\mu}+S_{21}-y^iq_i\,.
 \end{equation*}
 This coincides with
 \begin{equation*}
    \Phi_{31}^*[f_3]=f_3+S_{31}-z^{\mu}r_{\mu}=f_3+S_{32}+S_{21}-y^iq_i-z^{\mu}r_{\mu}\,,
 \end{equation*}
 by~\eqref{eq.S31},
 where $\Phi_{31}=\Phi_{32}\circ \Phi_{21}$\,.
\end{proof}

So far we have dealt with even functions and what we have defined as $\EThick$ will be called the \textbf{even microformal category}. Parallel constructions are based on the \emph{anticotangent} bundles, i.e., the cotangent bundles with reversed parity in the fibers (see~\cite{tv:nonlinearpullback}).   For local coordinates $x^a$ on a supermanifold $M$, let $x^*_a$ be  the conjugate antimomenta (fiber coordinates on $\Pi T^*M$). The canonical odd symplectic form on $\Pi T^*M$ is
\begin{equation}
 \o=d(dx^ax^*_a)=-(-1)^{\at}dx^adx^*_a=-(-1)^{\at}dx^*_adx^a\,,
\end{equation}
and let $-\Pi T^*M$ denote  $\Pi T^*M$   considered with the form $-\o$.

\begin{definition}\label{def.othick}
 An \emph{odd thick morphism} (or \emph{odd microformal morphism}) $\Psi\co M_1\oto M_2$ is specified by a formal odd generating function $S=S(x,y^*)$
 (locally defined) and corresponds to a formal
 canonical relation  $\Psi\subset \Pi T^*M_2\times (-\Pi T^*M_1)$ (denoted by the same letter)
 \begin{equation}\label{eq.orelphi}
    \Psi=\left\{\vphantom{\der{S}{x^a}(x,q)}\,(y^i,y^*_i\,;\, x^a,x^*_a)\ \right|\left. \ y^i= \der{S}{y^*_i}(x,y^*)\,, \ x^*_a=\der{S}{x^a}(x,y^*)\right\}\,.
 \end{equation}
\end{definition}
On the submanifold $\Psi$ we have
\begin{equation}
 dy^iy^*_i-dx^ax^*_a= d\bigl(y^iy^*_i-S\bigr)\,.
\end{equation}

Under changes of coordinates, the odd generating function $S$  of an odd thick morphism has the transformation law
\begin{equation}
    S'(x',y^{'*})= S(x,y^*)- y^iy^*_i + y^{i'}y^*_{i'}
\end{equation}
similar to~\eqref{eq.news}, where   variables in the r.h.s. are determined from the equations similar to those that arise in the even case.

The following theorems~\ref{thm.compodd},~\ref{thm.compoddass} and~\ref{thm.functorodd} are completely analogous to the ``even'' versions above and we omit their proofs.

\begin{theorem} \label{thm.compodd}
There is a well-defined  composition $\Psi_{32}\circ \Psi_{21}$  of odd thick morphisms, which is   an odd thick morphism
\begin{equation*}
    \Psi_{31}\co M_1\oto M_3
\end{equation*}
with the  generating function $S_{31}=S_{31}(x,z^*)$, where
\begin{equation}\label{eq.S31.odd}
    S_{31}(x,z^*)= S_{32}(y,z^*) + S_{21}(x,y^*) - y^iy^*_i
\end{equation}
and $y^i$ and $y^*_i$ are  expressed uniquely via $(x^a, z^*_{\mu})$  from the system
\begin{align}\label{eq.q31.odd}
    y^*_i&= \der{S_{32}}{y^i}\,(y,z^*)
\intertext{and}
 \label{eq.y31.odd}
    y^i&= \der{S_{21}}{y^*_i}\,(x,y^*)\,,
\end{align}
as    power series in $z^*_{\mu}$  and     functional power series in $S_{32}$. \qed
\end{theorem}

\begin{theorem} \label{thm.compoddass}
The composition of odd thick morphisms is associative. \qed
\end{theorem}

Odd thick morphisms form a formal category $\OThick$, which we call, the \textbf{odd microformal category}.  It is the formal neighborhood of the subcategory $\SMan\rtimes \pfunn$ contained as a closed subspace (and there are inclusion and retraction functors).   The affine action of the category $\SMan\rtimes \pfunn$ on supermanifolds of odd functions extends to a nonlinear action of the   formal category $\OThick$, as follows.

\begin{definition}
The \emph{pullback} $\Psi^*$  with respect to an odd thick morphism $\Psi\co M_1\oto M_2$ is  a formal mapping of functional supermanifolds
 \begin{equation}\label{eq.phistar1.odd}
    \Psi^*\co \pfunn(M_2) \to \pfunn(M_1)
 \end{equation}
defined for $\g\in \pfunn(M_2)$ by
\begin{equation}\label{eq.phistar2.odd}
    \Psi^*[\g] (x)= \g(y) + S(x,y^*) - y^iy^*_i\,,
\end{equation}
where $y^*_i$ and $y^i$ are determined from the equations
\begin{equation}\label{eq.phistarq.odd}
    y^*_i=\der{\g}{y^i}\,(y)
\end{equation}
and
\begin{equation}\label{eq.phistary.odd}
    y^i= \der{S}{y^*_i}(x,y^*)\,.
\end{equation}
\end{definition}

\begin{theorem} \label{thm.functorodd}
 For odd thick morphisms, the identity
 \begin{equation}
  (\Psi_{32}\circ \Psi_{21})^*=\Psi_{21}^*\circ \Psi_{32}^*\,
 \end{equation}
 holds. \qed
\end{theorem}

As in the even case, the pullback $\Psi^*$ is a formal nonlinear differential operator for which the $k$th term in the power expansion contains derivatives of orders $\leq k-1$. An analog of Theorem~\ref{thm.tphi} holds~\cite{tv:nonlinearpullback}. One can formulate   ``odd'' versions of Definition~\ref{def.nonlinhom} and Conjecture~\ref{conj.nonlinhom}.

\begin{remark} Pullback of functions with respect to a thick morphism is a particular case of the composition of thick morphisms (both in bosonic and fermionic cases)\,---\,the same as for usual pullbacks. One may wish to consider ``thick functions'' on supermanifolds as thick morphisms to $\RR$ or $\CC$. One may also wish to consider gluing of ``thick supermanifolds'' from ordinary ones with the help of thick diffeomorphisms  or, for example, to introduce ``thick analogs'' of Lie groups. Constructions in this section suggest many attractive paths which we hope to explore in the   future.
\end{remark}


\section{Application to vector bundles: the notion of the adjoint for a nonlinear map}\label{sec.adjoint}

In this section, we   generalize the notion of the adjoint of a linear operator. We   show that  using thick morphisms
one can speak of     the adjoint     for a nonlinear map  of vector spaces or vector bundles. Such generalized adjoints are thick morphisms rather than  ordinary maps. There are two versions of this construction, ``even'' and ``odd''.

Our construction is based on  the  canonical diffeomorphism between the cotangents of dual vector bundles discovered by Kirill~Mackenzie and Ping~Xu~\cite[Thm.~5.5]{mackenzie:bialg}\footnote{The special case of $E=TM$, i.e., the diffeomorphism $T^*TM\cong T^*T^*M$, is due to
Tulczyjew~\cite{tulczyjew:1977}; the case of general $E$ was  considered independently by J.-P.~Dufour in an unpublished  work.} (see also~\cite{mackenzie:diffeomorphisms},~\cite[Ch.~9]{mackenzie:book2005}; and \cite{tv:graded} for the super case)\,:
\begin{equation}
    T^*E\cong T^*E^*\,,
\end{equation}
which will be referred to  as the  {Mackenzie--Xu transformation}. (Some authors use the name `Legendre transformation', but this is  really confusing since the  Legendre transformation or transform  in the standard sense  acts on functions, not   points.)
There is a parallel canonical diffeomorphism for the fermionic case~\cite{tv:graded}
\begin{equation}
    \Pi T^*E\cong \Pi T^*(\Pi E^*)\,.
\end{equation}
Recall these natural diffeomorphisms in the form suitable for our purposes. For a vector bundle $E\to M$, denote   local coordinates on the base by $x^a$ and   linear coordinates in the fibers by $u^i$. The transformation law for $u^i$ has the form \new{$u^i=u^{i'}T_{i'}{^{i}}$}.
Denote the fiber coordinates for the dual bundle $E^*\to M$ and the antidual bundle $\Pi E^*\to M$ by $u_i$ and $\h_i$, respectively.  We assume that the invariant bilinear forms are $u^iu_i$ and $u^i\h_i$. (This means that $u_i$ and $\h_i$ are the right coordinates with respect to the basis which is `right dual' to a basis in $E$.) Consider the cotangent and the anticotangent bundles for $E$. Denote the canonically conjugate momenta for $x^a, u^i$ by $p_a,p_i$, and the conjugate antimomenta, by $x^*_a, u^*_i$. A similar notation will be used for $E^*$ and $\Pi E^*$.

The \emph{Mackenzie--Xu transformation}
\begin{equation}\label{bbb}
    \kir\co T^*E\to T^*E^*
\end{equation}
is defined by the formulas
\begin{equation}\label{eq.kir}
 \kir^*(x^a)=x^a\,, \ \kir^*(u_i)=p_i\,, \ \kir^*(p_a)=-p_a\,, \ \kir^*(p^i)=(-1)^{\itt}u^i\,.
\end{equation}
It is well-defined and is an antisymplectomorphism. (The choice of signs in~\eqref{eq.kir} agrees with that in book~\cite{mackenzie:book2005} and  differs from that of~\cite{tv:graded}. The choice used in~\cite{tv:graded} gives a symplectomorphism.)

An odd version of this transformation~\cite{tv:graded} (which we denote   by the same letter)
\begin{equation}
    \kir\co \Pi T^*E\to \Pi T^*(\Pi E^*)
\end{equation}
is defined by
\begin{equation}\label{eq.kirodd}
 \kir^*(x^a)=x^a\,, \ \kir^*(\h_i)=u^*_i\,, \ \kir^*(x^*_a)=-x^*_a\,, \ \kir^*(\h^{*i})= u^i\,.
\end{equation}
(note the absence of signs depending on parities). It is also an antisymplectomorphism  with respect to the canonical odd symplectic structures.

\begin{remark}
 The  invariance of   formulas~\eqref{eq.kir}, \eqref{eq.kirodd} is nontrivial and follows from the analysis of $T^*E$ and $\Pi T^*E$  as \emph{double vector bundles} over $M$. On the other hand, from the coordinate formulas~\eqref{eq.kir} and \eqref{eq.kirodd}, it is obvious that $\kir^*\o=-\o$   for the canonical symplectic structures. Moreover, one can immediately see that for the canonical Liouville $1$-forms
 \begin{equation}
  \kir^*(dx^ap_a+du_ip^i)=-(dx^ap_a+du^ip_i) +d(u^ip_i)
 \end{equation}
on the cotangent bundle and
\begin{equation}
  \kir^*(dx^ax^*_a+d\h_i\h^{*i})=-(dx^ax^*_a+du^iu^*_i) +d(u^iu^*_i)\,,
 \end{equation}
on  the anticotangent bundle. Note that 
$u^ip_i$ and $u^iu^*_i$ are invariant functions. 
\end{remark}

Now we proceed to constructing generalized adjoints.
Let $E_1$ and $E_2$ be vector bundles over a fixed base $M$. Consider a fiber map over $M$\,,
\begin{equation*}
    \Phi\co E_1\to E_2\,,
\end{equation*}
which is not necessarily fiberwise linear. (Here $\Phi$ is an ordinary map, not a thick morphism.) In local coordinates, it is given by
\begin{equation*}
   \Phi^*(y^a)=x^a\,, \  \Phi^*(w^{\a})=\Phi^{\a}(x,u)\,,
\end{equation*}
for some functions $\Phi^{\a}(x,u)$, where   $u^i$ and $w^{\a}$  are linear coordinates on the fibers of $E_1$ and $E_2$. For the fiber coordinates on the dual bundles we   use the  same letters  with the lower indices so that the forms $u^iu_i$ and $w^{\a}w_{\a}$ give the invariant pairings.

Note that it makes sense to speak about fiberwise thick morphisms. 

\begin{theorem} \label{thm.adjoint}
1. For an arbitrary  fiberwise map of vector bundles $\Phi\co E_1\to E_2$ over a base $M$, 
there are a fiberwise even thick morphism  \emph{(`adjoint')}
\begin{equation}
    \Phi^*\co E_2^*\tto E_1^*\,,
\end{equation}
and a  fiberwise odd thick morphism \emph{(`antiadjoint')}
\begin{equation}
    \Phi^{*\Pi}\co \Pi E_2^*\oto \Pi E_1^*\,,
\end{equation}
such that if  the map $\Phi\co E_1\to E_2$ is  fiberwise linear, i.e., a vector bundle homomorphism, then the thick morphisms $\Phi^*$ and  $\Phi^{*\Pi}$ are ordinary maps which are  the usual adjoint homomorphism and the adjoint homomorphism combined with the parity reversion, respectively.

2. For the composition of fiberwise maps of vector bundles over $M$\,,
\begin{equation}\label{eq.compos}
    E_1\fto{\Phi_{21}} E_2\fto{\Phi_{32}} E_3\,,
\end{equation}
we have the equality
\begin{equation}\label{eq.composadj}
    \bigl(\Phi_{32}\circ \Phi_{21}\bigr)^*= \Phi_{21}^*\circ \Phi_{32}^*\,,
\end{equation}
as   even thick morphisms $E_3^*\tto E_1^*$, and the equality
\begin{equation}\label{eq.composoddadj}
    \bigl(\Phi_{32}\circ \Phi_{21}\bigr)^{*\Pi}= \Phi_{21}^{*\Pi}\circ \Phi_{32}^{*\Pi}\,,
\end{equation}
as   odd thick morphisms $\Pi E_3^*\oto \Pi E_1^*$\,.
\end{theorem}
\begin{proof}  Consider a fiberwise  map\footnote{\,For clarity, although against our own taste,  we   use the physicists' notation with  arguments of functions equipped with   indices.} $\Phi\co E_1\to E_2$\,,
\begin{equation*}
    (x^a,u^i)\mapsto \bigl(y^a=x^a, w^{\a}=\Phi(x^a,u^i)\bigr)\,.
\end{equation*}
To the map $\Phi$  corresponds
the canonical relation $R_{\Phi}\subset T^*E_2\times (-T^*E_1)$\,,
\begin{equation*}
    R_{\Phi}=\left\{\vphantom{\der{S}{x^a}(x,q)}
    (y^a,w^{\a},q_i,q_{\a}\,;\, x^a,u^i, p_a, p_{i})\ \right|\left.\vphantom{\der{S}{x^a}(x,q)}
   (-1)^{\at}dq_ay^a+(-1)^{\att}dq_{\a}w^{\a}+ dx^ap_a+du^ip_i= dS\right\},
 \end{equation*}
with the generating function $S=S(x^a,u^i,q_i,q_{\a})$, where
\begin{equation}
    S= x^aq_a+ \Phi^{\a}(x^a,u^i)q_{\a}\,.
\end{equation}
We define the  thick morphism $\Phi^*\co E^*_2\tto E^*_1$   by a generating function $S^*=S^*(y^a,w_{\a}, p_a,p^i)$, where
\begin{equation}
    S^*:= y^ap_a+ \Phi^{\a}(y^a,(-1)^{\itt}p^i)w_{\a}\,.
\end{equation}
The corresponding canonical relation  $ {\Phi^*}\subset T^*E^*_1\times (-T^*E^*_2)$ is given by the equation
\begin{equation*}
    (-1)^{\at}dp_ax^a+(-1)^{\itt}dp^iu_i+dy^aq_a+dw_{\a}q^{\a}=dS^*\,,
\end{equation*}
 or, more explicitly,
\begin{equation*}
    x^a=y^a\,, \ u_i=\der{\Phi^{\a}}{u^i}(y,(-1)^{\itt}p^i)w_{\a}\,, \ q_a=p_a+\der{\Phi^{\a}}{x^a}(y,(-1)^{\itt}p^i)w_{\a}\,, \ q^{\a}=(-1)^{\att}\Phi^{\a}(y,(-1)^{\itt}p^i)\,.
\end{equation*}
The   construction of the thick morphism $\Phi^*$ can be stated  geometrically as follows.
We first    apply  the  transformation $\kir$ to the canonical relation $R_{\Phi}\subset T^*E_2\times (-T^*E_1)$. Since $\kir$ is an antisymplectomorphism, we obtain a Lagrangian submanifold $(\kir\times \kir)(R_{\f})\subset -T^*E^*_2\times T^*E^*_1$.
The thick morphism $\Phi^*$ is then defined by the opposite relation\,:
\begin{equation*}
    \Phi^*:=\bigl((\kir\times \kir)(R_{\f})\bigr)^{\text{op}}\subset T^*E^*_1\times (-T^*E^*_2)\,,
\end{equation*}
Expressing this by generating functions, we arrive at the formulas above. One can see that the thick morphism $\Phi^*\co E^*_2\tto E^*_1$ is \new {also} fiberwise. Let us check   that $\Phi^*$ is the ordinary adjoint when $\Phi\co E_1\to E_2$ is linear on fibers. Indeed, in such a case we have
\begin{equation*}
  \Phi(x^a,u^i)=u^i\Phi_i{}^{\a}(x)\,,
\end{equation*}
hence the above formulas give
\begin{equation*}
    x^a=y^a\,, \ u_i=\Phi_i{}^{\a}(y)w_{\a}\,
\end{equation*}
as expected.
The odd thick morphism
$\Phi^{*\Pi}\co \Pi E^*_2\oto \Pi E^*_1$
is built in a similar way: we take the canonical relation $R_{\Phi}\subset \Pi T^*E_2\times (-\Pi T^*E_1)$ corresponding to a map $\Phi\co E_1\to E_2$, apply the odd version of the Mackenzie--Xu transformation and then take the opposite relation.

To obtain equations~\eqref{eq.composadj} and ~\eqref{eq.composoddadj}, notice that  the composition of maps~\eqref{eq.compos} induces the composition of the corresponding canonical relations between the cotangent bundles  in the same order. This is preserved by the Mackenzie--Xu transformation. Taking the opposite relations reverses the order.
\end{proof}

\begin{corollary} \label{cor.pushforward}
On functions on the dual bundles, the pullback with respect to the
adjoint $ \Phi^*\co E^*_2\tto E^*_1$
induces a  \emph{`nonlinear pushforward map'}
 \begin{equation}
  \Phi_*:=(\Phi^*)^*\co \funn(E_1^*)\to \funn(E_2^*)\,.
 \end{equation}
The
restriction of $\Phi_*$ to  the space of even sections $\funn(M,E_1)$ regarded as the subspace in $\funn(E_1^*)$ consisting of the fiberwise-linear functions  takes it to the subspace $\funn(M,E_2)$ in $\funn(E_2^*)$ and coincides  with the usual pushforward  of sections $\Phi_*(\vp)=\Phi\circ \vp$.
\end{corollary}
\begin{proof}
 The nonlinear pushforward $ \Phi_* \co \funn(E_1^*)\to \funn(E_2^*)$ is defined as the pullback with respect to the thick morphism $\Phi^*\co E^*_2\tto E^*_1$.
 To an even function $f=f(x^a,u_i)$ the map $\Phi_*$ assigns the even function $g=\Phi_*[f]$, where $g(x^a,w_{\a})$  is given by:
 \begin{equation}
  g(x,w_{\a})=f(x,u_i)+ \Phi^{\a}\bigl(x,(-1)^{\itt}p^i\bigr)w_{\a}-u_ip^i\,,
  \end{equation}
and  $ u_i,  p^i$ are found from the equations
\begin{equation}
 p^i=\der{f}{u_i}(x,u_i)
\end{equation}
and
\begin{equation}
 u_i=\der{\Phi^{\a}}{u^i}\Bigl(x,(-1)^{\itt}\der{f}{u_i}(x,u_i)\Bigr)w_{\a}\,.
\end{equation}
The latter equation is solvable by iterations.
Now let the function $f$ on $E^*_1$ have the form $f(x,u_i)=v^i(x)u_i$, which corresponds to an even section $\vp=v^i(x)\ep_i$ of the bundle $E_1$. Then
\begin{equation}
 p^i=(-1)^{\itt}v^i(x)\,,
\end{equation}
hence
\begin{equation}
  \Phi_*[f]= v^i(x)u_i+ \Phi^{\a}\bigl(x,v^i(x)\bigr)w_{\a}-u_i(-1)^{\itt}v^i(x)= \Phi^{\a}\bigl(x,v^i(x)\bigr)w_{\a}\,,
  \end{equation}
which is the fiberwise linear function on $E^*_2$ corresponding to the   section $\Phi\circ \vp$\,.
\end{proof}

Similar statement holds for the odd case: there is an \emph{odd nonlinear pushforward} map  $\Phi_*^{\Pi}:=(\Phi^{*\Pi})^*$\,,
 \begin{equation}
  \Phi_*^{\Pi}\co \pfunn(\Pi E_1^*)\to \pfunn(\Pi E_2^*)\,.
 \end{equation}
On the space of even sections $\vp\in \funn(M,E_1)$ regarded as a subspace $\funn(M, E_1)\subset \pfunn(\Pi E_1^*)$ of the odd fiberwise-linear functions on $\Pi E_1^*$, the map $\Phi_*^{\Pi}$  again coincides with the obvious pushforward $\vp\mapsto \Phi\circ \vp$.

The algebra of fiberwise polynomial functions on the dual bundle $E^*$  is freely generated by the sections of $E$ over the algebra of functions on the base $M$. For the vector bundle homomorphisms   $E_1\to E_2$, the pushforward of functions $\fun(E_1^*)\to \fun(E_2^*)$ is the algebra homomorphism extending a linear map from free  generators. As   seen from Corollary~\ref{cor.pushforward}, the nonlinear pushforward map $\Phi_*\co \funn(E_1^*)\to \funn(E_2^*)$  can   be similarly regarded as the  extension of a nonlinear homomorphism from generators.

\begin{remark}
If the base $M$ is a point, we have a nonlinear map of vector spaces $\Phi\co V\to W$.  Replacing it by  Taylor expansion gives   a sequence of linear maps $\Phi_k\co S^kV\to W$.  The functions on the dual spaces can themselves be seen as elements of the symmetric powers. By expanding the pushforward $\Phi_*$ into a Taylor series, we arrive at linear maps of the form $S^n(\oplus S^pV)\to \oplus S^qW$. It would be interesting to obtain for them a  purely algebraic description.
\end{remark}

\begin{remark}
 From the proof of Theorem~\ref{thm.adjoint} it is clear that instead of an ordinary map one can start from a fiberwise even thick morphism $E_1\tto E_2$ and construct its adjoint  $E_2^*\tto E_1^*$ by the same method; or start from a fiberwise odd thick morphism $E_1\oto E_2$ and construct the antiadjoint $\Pi E_2^*\oto \Pi E_1^*$. 
\end{remark}

\begin{remark} `Nonlinear adjoints' can be   generalized   to vector bundles over different bases by using the concept of comorphisms of Higgins and Mackenzie~\cite{mackenzie_and_higgins:duality}\footnote{\,This notion has  a rich pre-history and numerous connections. Besides  citations in~\cite{mackenzie_and_higgins:duality}, see Guillemin and Sternberg~\cite{gs:geomasymp}, who suggested to redefine morphisms of vector bundles as, basically, comorphisms. A close notion  was introduced in~\cite{tv:class} in connection with integral transforms. In~\cite{cattaneo-dherin-weinstein:comorphisms} it is argued that comorphisms are   ``the correct'' notion in the context of Poisson geometry.}.  Suppose  $E_1\to M_1$ and  $E_2\to M_2$ are  fiber bundles over bases  $M_1$ and $M_2$. Then a \emph{bundle morphism} $\Phi\co E_1\to E_2$ can be defined as a fiberwise map over  the  fixed base $E_1\to \f^*E_2$ and a \emph{bundle comorphism} $\Psi\co E_1\to E_2$ can be defined as a fiberwise map over  the    fixed base  $\psi^*E_1\to  E_2$.  (It would be better to use for morphisms and comorphisms arrows of different shape.) In both cases, there is a map of the bases $\f$  or $\psi$
pointing in the same direction  for a morphism   and in the opposite direction  for a comorphism.

For bundles over the same base, morphisms and comorphisms over the identity map  coincide, and for manifolds regarded as `zero vector bundles', morphisms are  ordinary maps while comorphisms are    morphisms in the opposite category. As was shown in~\cite{mackenzie_and_higgins:duality}, for vector bundles (assuming the fiberwise linearity for maps over a fixed base), the adjoint of a morphism $E_1\to E_2$ is a comorphism $E^*_2\to E^*_1$ and vice versa; so  this gives an anti-isomorphism of the  two categories of vector bundles. To generalize this to our setup, one may wish to keep a map between the bases as an ordinary map while using fiberwise thick morphisms over a fixed base. This incorporates the possible nonlinearity of morphisms. In this way, one obtains base-changing `thick morphisms' and `thick comorphisms' of vector bundles to which the duality theory  extends.
\end{remark}

\section{Application to Lie algebroids and homotopy Poisson brackets}\label{sec.algebroid}

It is well known that, for   Lie algebras  $\mathfrak{g}_1$ and $\mathfrak{g}_2$,   a linear map of the underlying vector spaces $\f\co \mathfrak{g}_1\to \mathfrak{g}_2$
is a Lie algebra homomorphism  if the adjoint map of the dual spaces $\f^*\co \mathfrak{g}^*_2\to \mathfrak{g}^*_1$  is   Poisson with respect to the induced Lie--Poisson brackets (also known as the Berezin--Kirillov brackets). The same holds true
for Lie algebroids~\cite[Ch.10]{mackenzie:book2005} (see~\cite{mackenzie_and_higgins:duality} for base-changing morphisms). In this section we use the  construction of the adjoint for a nonlinear map of vector bundles and  the results from~\cite{tv:nonlinearpullback} to establish    the homotopy analogs of these statements for the case of $L_{\infty}$-morphisms of $L_{\infty}$-algebroids. It is convenient
to work in the super setting, though we generally suppress the prefix `super-'.

For simplicity consider the case of   fixed base. We do not consider the `if and only if' form of the statement. Our main theorem here is as follows.

\begin{theorem}\label{thm.linf}
 An $L_{\infty}$-morphism of $L_{\infty}$-algebroids over a base $M$ induces  $L_{\infty}$-morphisms of the homotopy Poisson and homotopy Schouten algebras of functions on the  dual and  antidual bundles respectively.
\end{theorem}

Before proving the theorem, we recall some definitions and statements.

Recall that an \emph{$L_{\infty}$-algebroid} (see, e.g.,~\cite{tv:higherpoisson}) is a \new{(super)} vector bundle $E\to M$ endowed with a sequence of $n$-ary \emph{brackets} that defines an $L_{\infty}$-algebra structure on sections and a sequence of $n$-ary \emph{anchors} $a\co E\times_M\ldots \times_M E\to TM$ (multilinear bundle maps) so that the brackets satisfy  the Leibniz identities with respect to the multiplication of sections by functions on the base,
\begin{equation}
 [u_1,\ldots,u_{n-1},fu_n]=a(u_1,\ldots,u_{n-1})(f)\,u_n + (-1)^{(\ut_1+\ldots +\ut_{n-1}+n)\ft}f\,[u_1,\ldots,u_n]\,.
\end{equation}
Here we   follow the convention of Lada and Stasheff for $L_{\infty}$-algebras~\cite{lada:stasheff} that the brackets are antisymmetric and of alternating parities. So that the unary bracket is odd, the binary bracket is even, etc. (Under the alternative convention,  all brackets are symmetric and odd; its equivalence with the antisymmetric convention is by the parity reversion, see   discussion in~\cite{tv:higherder}. In the sequel, we will need to use both versions.) With this convention, ordinary Lie algebroids are a particular case of $L_{\infty}$-algebroids. An $L_{\infty}$-algebroid structure on  $E\to M$ is equivalent to a formal  homological vector field on the supermanifold $\Pi E$. An  \emph{$L_{\infty}$-morphism} of $L_{\infty}$-algebroids $\Phi\co E_1 \infto E_2$ is   specified by a fiberwise map (in general, nonlinear) $\Phi\co \Pi E_1\to \Pi E_2$ such that the corresponding homological vector fields  are $\Phi$-related.\footnote{\,Note that here there is \textbf{no} single map of manifolds from $E_1$ to $E_2$; hence the nonstandard arrow $E_1 \infto E_2$ denoting a morphism.}   With some abuse of language, it is convenient to call the map $\Pi E_1\to \Pi E_2$ itself an $L_{\infty}$-morphism.  This definition includes as particular cases   $L_{\infty}$-morphisms of $L_{\infty}$-algebras and    morphisms of Lie algebroids. Note that what we call $L_{\infty}$-algebras  are often called `curved' $L_{\infty}$-algebras. By default we include a  $0$-ary operation.

An $L_{\infty}$-algebroid structure on  $E\to M$ induces a homotopy Poisson structure on the supermanifold $E^*$ and a homotopy Schouten structure on the supermanifold $\Pi E^*$. That means that there are given sequences of brackets turning the space $\fun(E^*)$ into an $L_{\infty}$-algebra in the Lada--Stasheff sense (``the antisymmetric convention'')  and $\fun(\Pi E^*)$ into an $L_{\infty}$-algebra in the sense of the alternative (``symmetric'') convention. Each bracket must also be a derivation in each argument. We shall refer to these brackets  as to the \emph{homotopy Lie--Poisson} and \emph{homotopy Lie--Schouten} brackets. These  structures on $E^*$ and $\Pi E^*$, as well as the homological vector field on $\Pi E$, are all equivalent to each other and should be seen as   different manifestations of one  structure of an $L_{\infty}$-algebroid, as  in the familiar cases of Lie algebras  and Lie algebroids~\cite{tv:graded,tv:qman-mack}.

\begin{proof}[Proof of Theorem~\ref{thm.linf}]
Consider a  {$L_{\infty}$-algebroid}   $E\to M$.    We shall give the proof for   the homotopy  Lie--Schouten brackets on  $\Pi E^*$. (The case of  the homotopy Lie--Poisson brackets on $ E^*$  is   similar.) Let $Q_E$ be the homological vector field on $\Pi E$ specifying the  algebroid structure in $E$. The homotopy Lie--Schouten brackets of functions on $\Pi E^*$  are  the  higher derived brackets generated by the odd  master Hamiltonian  $H^*=H^*_E$, i.e., an odd function on the cotangent bundle $T^*(\Pi E^*)$ satisfying   $(H^*,H^*)=0$ for the canonical Poisson bracket, which  is obtained  from the  fiberwise linear  Hamiltonian $H_E=Q_E\cdot p$ on $T^*(\Pi E)$ by the Mackenzie--Xu diffeomorphism $T^*(\Pi E)\cong T^*(\Pi E^*)$.
Suppose there is an $L_{\infty}$-morphism of $L_{\infty}$-algebroids $E_1\infto E_2$, i.e., a map $\Phi\co \Pi E_1\to \Pi E_2$ over $M$ such that the vector fields $Q_1$  and $Q_2$   are $\Phi$-related. This is equivalent to the Hamiltonians $H_1=H_{E_1}$ and $H_2=H_{E_2}$ being $R_{\Phi}$-related~\cite[Sec. 2, Ex.~6]{tv:nonlinearpullback}. By applying the Mackenzie--Xu transformations and flipping the factors, we conclude that the Hamiltonians $H_2^*=H_{E_2}^*$ and $H_1^*=H_{E_1}^*$ are $\Phi^*$-related, where $\Phi^*\co \Pi E_2^*\tto \Pi E_1^*$ is the adjoint thick morphism. By a key statement from~\cite{tv:nonlinearpullback} (Corollary from Theorems~6~and~7),  if the master Hamiltonians are  related by a thick morphism, then the   pullback  is an $L_{\infty}$-morphism of the homotopy Schouten algebras of functions. Hence the    pushforward map $\Phi_*=(\Phi^*)^*\co \funn(\Pi E^*_1)\to \funn(\Pi E^*_2)$  is an $L_{\infty}$-morphism, as claimed.
\end{proof}

With suitable modifications, the statement should hold for base-changing morphisms.

The following lemma should be known. It extends the corresponding property of ordinary Lie algebroids~\cite{mackenzie:book2005}. We give a proof for completeness (compare with the   statement for higher Lie algebroids~\cite{tv:napl,tv:qman}).

\begin{lemma}
 For an $L_{\infty}$-algebroid $E\to M$,  the higher anchors  assemble  to an $L_{\infty}$-morphism
\begin{equation*}
    a\co E\infto TM\,,
 \end{equation*}
to which we also refer as anchor (and use the same notation),  where $TM$ has the standard Lie algebroid structure.
\end{lemma}
\begin{proof}
The sequence of $n$-ary anchors assembles into a single map  $a\co \Pi E\to \Pi TM$, which is given by $a=\Pi Tp\circ Q$, where $Q=Q_E$ and
$\Pi Tp\co \Pi T(\Pi E)\to \Pi TM$ is the differential of the bundle projection  $p\co \Pi E\to M$. For an arbitrary $Q$-manifold $N$, the map $Q\co N\to \Pi TN$ is tautologically a $Q$-morphism, i.e., the vector fields $Q$ on $N$  and $d$ on $\Pi TN$   are $Q$-related. Also, for any map, its differential is a $Q$-morphism of the antitangent bundles. Hence the  map $a\co \Pi E\to \Pi TM$ is a $Q$-morphism as the composition of $Q$-morphisms. Therefore it gives an $L_{\infty}$-morphism $E\infto TM$ (which we denote by the same letter).
\end{proof}

\begin{corollary}
\label{coro1}
The anchor for an every $L_{\infty}$-algebroid $E\to M$ induces an  $L_{\infty}$-morphism
 \begin{equation}\label{eq.anchschout}
  a_*\co \funn(\Pi E^*)\to \funn(\Pi T^*M)
 \end{equation}
for the homotopy  Lie--Schouten brackets,
and an  $L_{\infty}$-morphism
\begin{equation} \label{eq.anchpoiss}
  a_*\co\pfunn(E^*)\to \pfunn(T^*M)
 \end{equation}
for the homotopy  Lie--Poisson brackets. (The functions on the bundles $\Pi T^*M$ and $T^*M$ are considered with the canonical Schouten and Poisson brackets, respectively.)
\end{corollary}

Note that at the right-hand sides of~\eqref{eq.anchschout} and~\eqref{eq.anchpoiss} there is only a binary bracket  while at the left-hand sides there are in general infinitely many brackets with all numbers of arguments. Therefore, for a general $L_{\infty}$-algebroid $E\to M$, these $L_{\infty}$-morphisms must be nontrivial, i.e., expressed by supermanifold maps that are  substantially  nonlinear.

\begin{corollary}
\label{coro2}
On a homotopy Poisson manifold $M$, there is an $L_{\infty}$-morphism
 \begin{equation}
  \funn(\Pi TM)\to \funn(\Pi T^*M)\,,
 \end{equation}
where functions on $\Pi TM$ 
are considered with the  higher Koszul brackets introduced in~\cite{tv:higherpoisson}.
\end{corollary}

To appreciate the meaning of Corollary~\ref{coro2}, recall that for an ordinary Poisson structure on a (super)manifold $M$, there is a linear transformation from forms to multivector fields, $\O^k(M)\to \Mult^k(M)$, preserving degrees  and parities, basically ``raising indices'' with the help of the Poisson tensor, which intertwines the de Rham differential on forms and the Poisson--Lichnerowicz differential on multivector fields, as well as the Koszul bracket on forms and the Schouten bracket on  multivector fields. Recall that the Poisson--Lichnerowicz differential $d_P$ can be defined by  $d_P=\lsch P, -\rsch$, the Schouten bracket with the Poisson tensor. The Koszul bracket induced by a Poisson structure can be defined on $1$-forms by formulas such as $[df,dg]_P=d\{f,g\}_P$ and $[df,g]_P=\{f,g\}_P$, where $\{f,g\}_P$ is a given Poisson bracket, and then extended to all forms as a biderivation. It is best to see this as a Lie algebroid structure induced on $T^*M$ (see~\cite{mackenzie:book2005}). For the homotopy case, the picture will be as follows~\cite{tv:higherpoisson}. A single binary Koszul bracket is replaced by an infinite sequence of `higher Koszul brackets' on $\O(M)$ making $T^*M$ an $L_{\infty}$-algebroid. It is still possible to define a linear transformation from forms to multivectors (no longer preserving degrees),   such that the diagram
\begin{equation*}
    \begin{CD} \Mult(M)@>{d_P}>> \Mult (M)\\
                @AAA         @AAA\\
                \O(M)@>{d}>> \O(M) \,,
    \end{CD}
\end{equation*}
is   commutative,
where the analog of the Poisson--Lichnerowicz differential $d_P=\lsch P, -\rsch$  is  an odd operator  (but   not of a particular degree). However, there is a problem with the brackets. Unlike the classical case, this linear map  $\O(M)\to \Mult(M)$ (and no linear map)  clearly cannot transform a sequence of  many higher Koszul brackets into one Schouten bracket. We conjectured in~\cite{tv:higherpoisson} that an $L_{\infty}$-morphism from $\O(M)$ to $\Mult(M)$    must exist  instead. Corollary~\ref{coro2} gives the desired solution. The linear map from forms to multivectors constructed in~\cite{tv:higherpoisson} is induced by a fiberwise (nonlinear) map $\Pi T^*M\to \Pi TM$, which represents the anchor $T^*M\infto TM$.  The dual to it   is a thick morphism $\Pi T^*M\tto \Pi TM$,   the nonlinear  pullback by which is   exactly the sought-for $L_{\infty}$-morphism. See~\cite{tv:linfbialg} for details.

\begin{corollary}[generalization of Corollary~\ref{coro2}]
\label{coro3}
 There is an $L_{\infty}$-morphism of homotopy Schouten algebras
 \begin{equation}
  \funn(\Pi E)\to \funn(\Pi E^*)\,
 \end{equation}
 for `triangular $L_{\infty}$-bialgebroids'.
\end{corollary}

Recall that Mackenzie and Xu~\cite{mackenzie:bialg} introduced  the concept of a triangular Lie bialgebroid as a  generalization of Drinfeld's triangular Lie bialgebras.  It is a pair of vector bundles in duality $(E,E^*)$, where a Lie algebroid structure on $E$ is initially given and the bundle $E^*$ is made  a Lie algebroid with the help of an element $r\in \Gamma(M, \Lambda^2E)$ playing the role of the classical $r$-matrix. In our language, $r$ is a fiberwise quadratic function on $\Pi E^*$. The Lie algebroid structure on $E^*$ is  defined by the Hamiltonian $H_{E^*}:=(H^*_E,r)\in \fun(T^*(\Pi E^*))$, where $H^*_E$ is obtained by the Mackenzie--Xu transformation from the Hamiltonian   $H_E\in \fun(T^*(\Pi E))$ corresponding to the Lie algebroid structure on $E$. (Counting weights shows that the Hamiltonian  $H_{E^*}$ is linear in momenta on $\Pi E^*$ as required.) The pair $(TM,T^*M)$ for a Poisson manifold
is a model example of a triangular Lie bialgebroid.  The role of an `$r$-matrix' is played by the Poisson bivector. Transporting this analogy to  the  homotopy case, we can define an $L_{\infty}$ analog of the Mackenzie--Xu triangular Lie bialgebroids. For a pair $(E,E^*)$, one starts from an $L_{\infty}$-algebroid structure on $E$ and an even function $r$ on $\Pi E^*$ (no constraints on degrees), and then introduces a compatible `triangular'  structure, which will make the pair $(E,E^*)$ a \emph{triangular $L_{\infty}$-bialgebroid}.\footnote{ There is some freedom as to what should be called  an \emph{$L_{\infty}$-bialgebroid} structure on $(E,E^*)$ in general.  The options range from   $L_{\infty}$-algebroid structures on $E$ and $E^*$  with a compatibility condition expressible as $(H_1,H_2)=0$ for the corresponding   odd Hamiltonians which live  on $T^*(\Pi E) \cong T^*(\Pi E^*)$, which in particular gives a self-commuting Hamiltonian $H:=H_1+H_2$,  to the apparently more general structure described by a single self-commuting odd Hamiltonian $H$ of an  arbitrary form. In the latter option  the distinction between the bialgebroid and its `Drinfeld double' looks as    blurred. One should certainly wish to have  a pair of structures that can be combined into a family.} The key observation here is that the homotopy analog of a triangular structure is the shift in the argument of the master Hamiltonian, $H(x,p)\mapsto H'(x,p)=H(x, p+\der{r}{x})$. Corollary~\ref{coro3} in this setting arises as an abstract version of Corollary~\ref{coro2}. We
elaborate these questions elsewhere (see~\cite{tv:homotopytriang} and \cite{tv:linfbialg}).

\section{Quantum thick morphisms: general properties} \label{sec.quantum}

We shall show now  that the construction  of thick morphisms in the bosonic case  has a ``quantum'' counterpart. Namely, we shall define ``quantum pullbacks'' depending on Planck's constant $\hbar$   as certain oscillatory integral operators
that transform  functions on one (super)manifold  to functions on another (super)manifold. We then   define the ``quantum microformal category'' as the dual   to the   category of such integral operators. We shall show that in the limit $\hbar\to 0$ this picture gives rise to thick morphisms and the corresponding nonlinear pullbacks.  This in hindsight may be seen as clarifying the origin  of the above ``classical'' constructions. For quantum pullbacks it is possible to give a closed formula  as opposed to the pullbacks by   ``classical'' thick morphisms,  defined only by an iterative procedure.   Quantum thick morphisms were first introduced in the short note~\cite{tv:oscil}. Some further results were obtained in preprint~\cite{tv:qumicro}.

We first need to introduce a class of functions on which quantum pullbacks will be acting.

\begin{definition} An \emph{oscillatory wave function} on a (super)manifold $M$ is a linear combination of  formal expressions of the form
\begin{equation}\label{eq.oscfun}
    w_{\hbar}(x)=a(\hbar, x)e^{\frac{i}{\hbar}b(x,\hbar)}
\end{equation}
where $a(x,\hbar)=\sum_{n\geq 0}\hbar^n a_n(x)$ and $b(x,\hbar)=\sum_{n\geq 0}\hbar^n b_n(x)$ are formal  expansions in non-negative powers   of $\hbar$ whose coefficients are smooth functions on $M$. (Here $b(x,\hbar)$ is even.)
\end{definition}

We customarily drop  explicit  indication of depending on $\hbar$ for oscillatory wave functions and write $w(x)$ for $w_{\hbar}(x)$. We assume natural rules of manipulations with the expression like~\eqref{eq.oscfun}. Note   that
we can always re-arrange the exponential    in~\eqref{eq.oscfun}
so to make    $w(x)=A(\hbar, x)e^{\frac{i}{\hbar}b_0(x)}$, with     no dependence on $\hbar$ in the phase $b_0(x)$.  Conversely, an invertible factor in front of the exponential can be made into a term in the phase if we  forsake the  reality restriction.
Oscillatory wave functions on $M$ form  an algebra, which   we denote  $\ofun(M)$ and  which extends    the algebra
$\funh(M):=\fun(M)[[\hbar]]$ of  formal power series in $\hbar$ with smooth coefficients\,.  Symbolically,
\begin{equation*}
    \ofun(M)=\funh(M) \exp \frac{i}{\hbar}\fun(M)\,.
\end{equation*}

Consider supermanifolds $M_1$ and $M_2$. In the same way as  thick morphisms $M_1\tto M_2$ are specified by their generating functions,  quantum thick morphisms will be specified by certain ``quantum'' generating functions.
As in Section~\ref{sec.categories}, denote by $x^a$   local coordinates on $M_1$, by $y^i$   local coordinates on $M_2$, and by $p_a$ and $q_i$ denote the corresponding conjugate momenta.
In  given coordinate systems on $M_1$ and $M_2$, a   \emph{quantum generating function} $S_{\hbar}(x,q)$ is a formal power series in $q_i$\,,
\begin{equation}
    S_{\hbar}(x,q)=S_{\hbar}^0(x)+ \f^i_{\hbar}(x)q_i + \frac{1}{2}\,S^{ij}_{\hbar}(x)q_jq_i + \frac{1}{3!}\,S^{ijk}_{\hbar}(x)q_kq_jq_i + \ldots
\end{equation}
where the coefficients   are formal power series in $\hbar$. Note that, the same as for the ``classical'' case considered before, $S_{\hbar}(x,q)$ is a coordinate representation of a geometric object and   not a scalar function. Its transformation law will be clarified shortly.

\begin{definition} A  \emph{quantum   thick morphism} or   \emph{quantum   microformal morphism}
\begin{equation*}
    \hat\Phi\co M_1\ttoq  M_2
\end{equation*}
with a (quantum) generating function $S_{\hbar}(x,q)$ is identified with its action on functions
\begin{equation*}
    \hat\Phi^*\co \ofun(M_2)\to \ofun(M_1)
\end{equation*}
in the opposite direction, called \emph{quantum pullback} and  defined  by the formula
\begin{equation}\label{eq.phihat}
    (\hat\Phi^* w)(x)= \int_{T^*M_2} Dy\,\Dbar q \,\, e^{\frac{i}{\hbar}\left(S_{\hbar}(x,q)-y^iq_i\right)}\,w(y)\,.
\end{equation}
Integration in~\eqref{eq.phihat} is with respect to the normalized Liouville measure $Dy\,\Dbar q$ on $T^*M_2$. Here and in the future, we use  the notation $\Dbar q:=(2\pi\hbar)^{-n}(i\hbar)^{m}Dq$   if $Dq$ is a coordinate volume element  in  dimension   $n|m$.
\end{definition}

The source of the normalization factor above is in the formulas for the direct and inverse $\hbar$-Fourier transform. Recall that    on $\R{n|m}$ they read
\begin{equation*}
    \tilde f(p)=\int Dx\; e^{-\frac{i}{\hbar}x^ap_a}f(x)
\end{equation*}
and
\begin{equation*}
      f(x)=
     (2\pi\hbar)^{-n}(i\hbar)^{m} \int Dp\; e^{\frac{i}{\hbar}x^ap_a}\tilde f(p)=
     \int \Dbar p\; e^{\frac{i}{\hbar}x^ap_a}\tilde f(p)\,,
\end{equation*}
where the integration is over $\R{n|m}$   and over its dual. (There may be  an extra common sign factor depending on choices of signs in the Berezin integral, which we take to be   $1$.)  In particular,
\begin{equation*}
      \delta(x)= \int \Dbar p\; e^{\frac{i}{\hbar}x^ap_a}\,.
\end{equation*}

\begin{remark} The form of integral operators~\eqref{eq.phihat} is familiar in the theory of partial differential equations (not   the super case of course). Operators of a slightly more general form
\begin{equation}\label{eq.egor}
    Au(x)= \int e^{i\left(S(x,p)-x'p\right)}a(x,p) \,  u(x') \,dx'\dbar p \,,
\end{equation}
but with both  $x$ and $x'$ in the same   domain  $\Omega\subset \R{n}$,
were studied by M.~I.~Vishik and G.~I.~\`{E}skin~\cite{vishik:eskin-conv1965} and especially   by Yu.~V.~Egorov~\cite{egorov:canonicpdo1969,egorov:canonicpdo1971} and M.~V.~Fedoryuk~\cite{fedoryuk:pdo1971}.  Together with   Maslov's  canonical operator~\cite{maslov:pert1965}, they were precursors of the
Fourier integral operators   introduced by L.~H\"{o}rmander in~\cite{hoer:spectralofell-1968} and \cite{hoer:fio1-1971}. H\"{o}rmander~\cite{hoer:fio1-1971} stressed as   crucial observation of Egorov  a connection between  operators~\eqref{eq.egor} and canonical transformations in $T^*\Omega$. (As noticed in Fedoryuk~\cite{fedoryuk:pdo1971}, such a connection was  indicated earlier by~V.~A.~Fock~\cite{fock:canon1959-69}, who  was making a precise statement out of Dirac's analogy   between unitary transformations in quantum mechanics and canonical transformations in classical mechanics. See also~Fock~\cite[Ch.~III \S 16]{fock:book}.)
In H\"{o}rmander's construction of Fourier integral operators,  canonical transformations gave way to canonical relations,  specified by equivalence classes of   phase functions depending on auxiliary variables. In   standard theory, these canonical relations are conical, so the  phase functions are positively-homogeneous of degree $+1$~\cite{treves:introtofio}, \cite{egorov:microlocalencyclo}, \cite{shubin:book}. Operators~\eqref{eq.phihat} can be seen therefore as a special case of Fourier integral operators, but not exactly fitting in the standard definitions because of the different type of their phase functions.
\end{remark}

\begin{example} \label{ex.s0phiquant}
Let $S_{\hbar}(x,q)=S_{\hbar}^0(x)+\f_{\hbar}^i(x)q_i$. Then
\begin{equation}\label{eq.intpullback}
    (\hat\Phi^* w)(x)= e^{\frac{i}{\hbar}S_{\hbar}^0(x)}\int_{T^*M_2} \Dbar(y,q) \,\, e^{\frac{i}{\hbar}(\f_{\hbar}^i(x)-y^i)q_i}\,w(y)=
    e^{\frac{i}{\hbar}S_{\hbar}^0(x)}w(\f_{\hbar}(x))\,.
\end{equation}
We arrive at a ``quantum analog'' of the category $\SMan\rtimes \funn$ and  its action on smooth functions. Morphisms here   are pairs $(\f, e^{\frac{i}{\hbar}f})$ and the composition of pairs is given by $(\f_{32},e^{\frac{i}{\hbar}f_{32}})\circ (\f_{21},e^{\frac{i}{\hbar}f_{21}})=(\f_{32}\circ \f_{21}, e^{\frac{i}{\hbar}(\f_{21}^*f_{32}+f_{21})})$. The phase functions $f$ and maps $\f$ are expansions in nonnegative powers of    $\hbar$, $f =f_{\hbar}$  and  $\f=\f_{\hbar}$   (so the ``maps'' are formal perturbations of ordinary maps). The action~\eqref{eq.intpullback} is clearly well-defined for oscillatory wave functions $w$.
\end{example}

\begin{example}
Let $w(y)\equiv 1$. Then, for arbitrary $S_{\hbar}(x,q)$, we have
\begin{equation}\label{eq.qpullbofone}
    \hat\Phi^* (1) = e^{\frac{i}{\hbar}S_{\hbar}^0(x)}\,.
\end{equation}
\end{example}

\begin{example} \label{ex.test}
Let $w(y)=e^{\frac{i}{\hbar}y c}$, where $yc\equiv y^ic_i$ and $c_i$ are parameters. Then
\begin{equation}\label{eq.qpullataexp}
    (\hat\F^*w)(x)= e^{\frac{i}{\hbar}S_{\hbar}(x,c)}\,.
\end{equation}
Compare Example~\ref{ex.testfunction}\,.  We can restate this as a formula for reconstructing the quantum generating function:
\begin{equation}\label{eq.qugenfun}
    e^{\frac{i}{\hbar}S_{\hbar}(x,q)} = \hat\F^*\!\bigl[e^{\frac{i}{\hbar}y q}\bigr](x)
\end{equation}
(where we restored $q$ in the argument).
\end{example}

Let $S_0(x,q)$ be obtained by   substituting $\hbar=0$ in  a quantum generating function $S_{\hbar}(x,q)$. We regard $S_0(x,q)$ as the generating function of a classical thick morphism $\Phi\co M_1\tto M_2$. Before we have clarified the transformation law for quantum generating functions, this would make sense at least in a fixed coordinate system. We shall write $\F=\lim\limits_{\hbar\to 0}\hat\F$\,.

\begin{theorem}  \label{thm.quantstatph}
In the limit $\hbar\to 0$,  the quantum pullback $\hat\Phi^*$   transforms the phase  of an oscillatory wave function   as the pullback $\Phi^*$ by the classical thick morphism $\F=\lim\limits_{\hbar\to 0}\hat\F$, so that if $w(y)=e^{\frac{i}{\hbar}g(y)}$ on $M_2$, then
\begin{equation*}
    (\hat\Phi^*w)(x)= e^{\frac{i}{\hbar}f_{\hbar}(x)}\,,
\end{equation*}
on $M_1$, where
\begin{equation*}
    f_{\hbar}=\Phi^*[g]+O(\hbar)\,.
\end{equation*}
\end{theorem}
\begin{proof}
For a wave function $w(y)=e^{\frac{i}{\hbar}g(y)}$, we have
\begin{equation*}
    (\hat\Phi^* w)(x)= \int_{T^*M_2} \Dbar(y,q) \,\, e^{\frac{i}{\hbar}\left(S_{\hbar}(x,q)-y^iq_i +g(y)\right)}\,.
\end{equation*}
By the stationary phase method (see   Appendix~\ref{sec.append}), the value of the integral, in the main order in $\hbar$, is the  exponential evaluated at the critical points of the phase when $\hbar\to 0$. By differentiating with respect to $y^i$ and $q_i$ and setting the result to zero, we arrive at the system of equations
 \begin{equation*}
    q_i=\der{g}{y^i}(y)\,,\quad y^i=(-1)^{\itt}\der{S_0}{q_i}(x,q)
 \end{equation*}
for determining $y^i, q_i$, the unique solution of which should be substituted into $S_{0}(x,q)+g(y)-y^iq_i$ to obtain a function $f(x)$ as the leading term of the phase. These are exactly   equations~\eqref{eq.phistarq} and \eqref{eq.phistary} in the definition of pullback, and  $f=\Phi^*[g]$ as claimed.
\end{proof}

\begin{remark} The stationary phase method~\cite{fedoryuk:pdo1971} can be applied   to $\hat\Phi^*w$ for $w=a(x,\hbar)e^{\frac{i}{\hbar}g(x)}$   and it also allows to find all terms in the expansion in $\hbar$ (at least, their general form), not only the main term. The fact that quantum pullback preserves the class of oscillatory wave functions follows from here.
Note that the square root of the Hessian arising as a factor in the stationary phase method  can be formally subsumed into the phase   as a  correction of the first order in $\hbar$. Also note that since in the main order the quantum pullback reduces to the classical pullback, which is a formal map, so is the quantum pullback (formal on the phases). For convenience, we included the precise statements concerning the stationary phase method in the form suitable for our needs in the appendix. See Theorems~\ref{thm.varfedoryuk} and \ref{thm.ourstatphase} there.
\end{remark}

Integral~\eqref{eq.phihat} actually can be solved in a closed form, giving an expression  for a quantum pullback $\hat \Phi^*\co \ofun(M_2)\to \ofun(M_1)$ as a   ``formal differential operator''. (This is an advantage over  pullbacks by classical thick morphisms,  given in general   only by an iterative procedure.)   Let us write a quantum generating function $S_{\hbar}(x,q)$ defining a quantum microformal morphism $\hat\Phi\co M_1\ttoq M_2$ in the form similar to~\eqref{eq.sexpandgroup2},
\begin{equation}\label{eq.qgenfun2}
    S_{\hbar}(x,q)=S_{\hbar}^0(x)+\f^i_{\hbar}(x)q_i+S^{+}_{\hbar}(x,q)\,,
\end{equation}
where   $S^{+}_{\hbar}(x,q)$ is the sum of all  terms of order $\geq 2$ in $q_i$. 

\begin{theorem} \label{thm.quantactionexp}
The action of $\hat \Phi^*\co \ofun(M_2)\to \ofun(M_1)$
can be expressed as follows:
\begin{equation}\label{eq.phiw}
     \bigl(\hat \Phi^*w\bigr)(x)=e^{\frac{i}{\hbar}S_{\hbar}^0(x)}
    \left(e^{\frac{i}{\hbar}S^{+}_{\hbar}\left(x,\frac{\hbar}{i}\der{}{y}\right)}w(y)\right)_{\left|\vphantom{\int\limits_a^b}\ y^i=\f^i_{\hbar}(x)\right.}\,.
\end{equation}
It is  a formal differential operator  over a map 
$\f_{\hbar}\co M_1\to M_2$ given by  $y^i=\f^i_{\hbar}(x)$.
\end{theorem}
\begin{proof} Substituting~\eqref{eq.qgenfun2} into~\eqref{eq.phihat} gives
\begin{multline*}
    (\hat\Phi^* w)(x)=\int \Dbar(y,q) \,\, e^{\frac{i}{\hbar}(S_{\hbar}^0(x)+\f^i_{\hbar}(x)q_i+S^{+}_{\hbar}(x,q)-y^iq_i)}\,w(y)=\\
    e^{\frac{i}{\hbar}S_{\hbar}^0(x)}  \int  \Dbar q \,\,
    e^{\frac{i}{\hbar}\f^i_{\hbar}(x)q_i}
    e^{\frac{i}{\hbar}S^{+}_{\hbar}(x,q)}
    \int Dy \,\,
    e^{-\frac{i}{\hbar}y^iq_i}
    \,w(y)\,.
\end{multline*}
The integral  is the composition of the ($\hbar$-)Fourier transform of a function $w(y)$ from  the variables $y^i$ to the variables $q_i$, the multiplication by $e^{\frac{i}{\hbar}S^{+}_{\hbar}(x,q)}$, treated as a function of $q_i$  with $x^a$ seen as parameters, and the inverse  Fourier transform from $q_i$ to $y^i$, where $\f^i_{\hbar}(x)$ is substituted for $y^i$, followed finally by the multiplication by the phase factor $e^{\frac{i}{\hbar}S_{\hbar}^0(x)}$. Recalling the standard relation between multiplication and differentiation under   Fourier transform, we arrive  at the claimed result.
\end{proof}

\begin{remark} The notion of a differential operator over a smooth map  as opposed to operators on a single manifold is  not very standard, but should be self-explanatory. 
Separating   the ``differentiation part'' such as $S^{+}_{\hbar}\left(x,\frac{\hbar}{i}\der{}{y}\right)$ from the purely ``substitution part'' $y^i=\f^i_{\hbar}(x)$ in~\eqref{eq.phiw}  is of course coordinate-dependent. Naively, there are   three ingredients in $\hat\Phi^*$: a differential operator of   infinite order in $y^i$ and of the form $1+O(\hbar)$  in $\hbar$  (starting with the second derivatives and where each term with the derivatives of order $k$ is of order $k-1$ in $\hbar$), the substitution as such, and the multiplication by the phase factor. Thus, a general quantum thick morphism $\hat\Phi$ can be seen as a   perturbation due to the term $S^{+}_{\hbar}(x,q)$ in the expansion~\eqref{eq.qgenfun2} of the generating function, of a morphism of the form $(\f_{\hbar}, e^{\frac{i}{\hbar}f_{\hbar}})$ as in Example~\ref{ex.s0phiquant}.
\end{remark}

We can push this a bit further by noticing that the quantum pullback $\hat\F^*$ can be written as an integral operator
\begin{equation}\label{eq.qupullbviak}
    (\hat\F^*w)(x)=\int\!  Dy\, K(x,y)\, w(y)
\end{equation}
with the Schwarz kernel
\begin{equation}\label{eq.kernofhatphistar}
    K(x,y)=\int\! \Dbar q\, e^{\frac{i}{\hbar}\left(S_{\hbar}(x,q)-y^iq_i\right)}\,,
\end{equation}
i.e., the $\hbar$-Fourier transform (up to a factor) of the function $e^{\frac{i}{\hbar} S_{\hbar}(x,q)}$ from $q$ to $y$\,. By expanding $S_{\hbar}(x,q)$ as in~\eqref{eq.qgenfun2} and using manipulations similar to the proof of Theorem~\ref{thm.quantactionexp}, we can express the integral  kernel of the operator $\hat\F^*$ as
\begin{equation}\label{eq.kernelphi}
     K(x,y)= e^{\frac{i}{\hbar}S_{\hbar}^0(x)}
     e^{\frac{i}{\hbar}S^{+}_{\hbar}\left(x,- \frac{\hbar}{i}\der{}{y}\right)}\,\delta\bigl(y-\f_{\hbar}(x)\bigr)
\end{equation}
(note the minus sign in the argument of $S^{+}_{\hbar}$). This basically is a re-statement of Theorem~\ref{thm.quantactionexp}. In this form it is clear that the integral kernel of $\hat\F^*$  is supported on a formal neighborhood of the graph of the ``$\hbar$-perturbed'' map $\f_{\hbar}\co M_1\to M_2$\,.


\begin{theorem} \label{thm.qucomp}
The composition of quantum thick morphisms $ M_1\tttoq{\hat\Phi_{21}} M_2\tttoq{\hat\Phi_{32}} M_3$
with generating functions $S_{21}(x_1,p_2)$ and $S_{32}(x_2,p_3)$ is a quantum thick morphism $ M_1\tttoq{\hat\Phi_{31}}   M_3$
with the generating  function $S_{31}(x_1,p_3)$ given by
\begin{equation}\label{eq.qucomp}
    e^{\frac{i}{\hbar}S_{31}(x_1,p_3)}=\int_{T^*M_2} \Dbar (x_2, p_2)\, e^{\frac{i}{\hbar}\left(S_{32}(x_2,p_3)+ S_{21}(x_1,p_2)-x_2p_2\right)}\,.
\end{equation}
\emph{(Here $S_{21} : =S_{21,\hbar}$, etc.; we suppress $\hbar$ for the simplicity of notation.)}
In the limit $\hbar\to 0$, this composition law becomes the composition law for classical thick morphisms given by Theorem~\ref{thm.classcompos}.
\end{theorem}
\begin{proof}
Apply the composition $\hat\Phi_{21}^*\circ \hat\Phi_{32}^*$ to a `test function' $w(x_3)=e^{\frac{i}{\hbar}x_3p_3}$ (see Example~\ref{ex.test}).  The claim is that the result is an oscillatory exponential of the desired form.  We work in the abbreviated notation and denote coordinates on the manifolds  $M_i$ by $x_i$ and the conjugate momenta by $p_i$, where $i=1,2,3$. We have
\begin{multline*}
    \hat\F_{21}^*\left(\hat\F_{32}^*\bigl[e^{\frac{i}{\hbar}x_3p_3}\bigr]\right)(x_1)=
    \hat\F_{21}^*\bigl[ \hat\F_{32}^*\bigl[e^{\frac{i}{\hbar}x_3p_3}\bigr](x_2)\bigr](x_1)=\\
    \int Dx_2\,\Dbar p_2\; e^{\frac{i}{\hbar}\left(S_{21}(x_1,p_2)-x_2p_2\right)}
    \int Dx_3 \,\Dbar p_3'\; e^{\frac{i}{\hbar}\left(S_{32}(x_2,p_3')-x_3p_3'\right)} e^{\frac{i}{\hbar}x_3p_3}=\\
    \int Dx_2\,\Dbar p_2\, Dx_3 \,\Dbar p_3'\; e^{\frac{i}{\hbar}\left(S_{21}(x_1,p_2)+S_{32}(x_2,p_3')-x_2p_2+x_3(p_3-p_3')\right)}=  \\
    \int Dx_2\,\Dbar p_2\; e^{\frac{i}{\hbar}\left(S_{21}(x_1,p_2)+S_{32}(x_2,p_3)-x_2p_2\right)}\,.
\end{multline*}
From the stationary phase method (see  Theorem~\ref{thm.ourstatphase} of the appendix) we observe, first, that the latter integral can be written as an exponential $e^{\frac{i}{\hbar}S_{31}(x_1,p_3)}$ for some   function
$S_{31}$ depending on $\hbar$; and, secondly, that in the limit $\hbar\to 0$, which is indicated by $0$ in the subscripts, we should have
\begin{equation*}
    S_{31,0}(x_1,p_3) = S_{21,0}(x_1,p_2)+S_{32,0}(x_2,p_3)-x_2p_2
\end{equation*}
where the   variables $x_2$ and $p_2$ are found from the equations
\begin{equation*}
    x_2^i=(-1)^{\itt}\der{S_{21,0}}{p_{2\,i}}(x_1,p_2)\,, \quad p_{2\,i}=\der{S_{32,0}}{x_2^i}(x_2,p_3)\,.
\end{equation*}
This is exactly the composition law for classical generating functions  as given by Theorem~\ref{thm.classcompos}\,.
\end{proof}

\begin{theorem}[transformation law for quantum generating functions] \label{thm.qutransf}
Let $x^a=x^a(x')$, $y^i=y^i(y')$ and  $x^{a'}=x^{a'}(x)$, $y^{i'}=y^{i'}(y)$ be mutually inverse changes of local coordinates on $M_1\times M_2$. Then quantum generating functions $S_{\hbar}(x,q)$ and $S'_{\hbar}(x',q')$ specifying the same quantum thick morphism $\hat\F\co M_1\ttoq M_2$ in coordinate systems $x,y$ and $x',y'$ are related by the transformation law
\begin{equation}\label{eq.translawqu}
    e^{\frac{i}{\hbar}S'_{\hbar}(x',q')}=\int Dy \,\Dbar q \, e^{\frac{i}{\hbar}\bigl(S_{\hbar}\left(x(x'),q\right)-yq+y'(y)q'\bigr)}\,,
\end{equation}
where we use abbreviated notation such as $yq\equiv y^iq_i$\,.
\end{theorem}
\begin{proof} Similarly to the proof of Theorem~\ref{thm.qucomp},   apply $\hat \F^*$, for a quantum thick morphism $\hat \F$ specified by   $S_{\hbar}(y,q)$ in the `old' coordinates $x^a,y^i$, to a  test function  $w=e^{\frac{i}{\hbar}y^{i'}q_{i'}}$, where $y^{i'}$ are the `new' coordinates on $M_2$ and $q_{i'}$ are the conjugate momenta, expressing the result also via the `new' coordinates $x^{a'}$ on $M_1$. We obtain
\begin{equation*}
    \hat \F^* [e^{\frac{i}{\hbar}y'q'}]= \int Dy\,\Dbar q\; e^{\frac{i}{\hbar}(S(x,q)-yq)}\, e^{\frac{i}{\hbar}y'(y)q'}=
    \int Dy\,\Dbar q\; e^{\frac{i}{\hbar}\left(S(x,q)-yq+y'(y)q'\right)} \,,
\end{equation*}
where it remains to substitute $x=x(x')$. The integral is of the type covered by Theorem~\ref{thm.ourstatphase} in the appendix and we may conclude that it equals to an oscillating exponential of the form $e^{\frac{i}{\hbar}S'_{\hbar}(x',q')}$, which therefore gives the quantum generating function of the morphism $\hat \F$ in the `new' coordinates on $M_1$ and $M_2$ expressed by~\eqref{eq.translawqu}, as claimed.
\end{proof}

(This included the independence of the notion of a quantum thick morphism of a choice of coordinates.)

If we apply the stationary phase method to the integral in the right-hand side of~\eqref{eq.translawqu}, we will arrive at the equations
\begin{equation*}
    y^i=(-1)^{\itt}\der{S_{\hbar}}{q_i}(x(x'), q)\,, \quad q_i=\der{y^{i'}}{y^i}(y)q_{i'}
\end{equation*}
for determining $y^i$ and $q_i$ (as functions of $x'$ and $q'$) at the stationary point\,. Then
\begin{equation*}
    S'_{\hbar}(x',q')= S_{\hbar}\left(x(x'),q\right)-yq+y'(y)q' + O(\hbar)\,.
\end{equation*}
Hence in the limit $\hbar\to 0$, the transformation law for quantum generating functions $S_{\hbar}$ becomes, as anticipated, the transformation law~\eqref{eq.news} for classical generating functions $S$, $S=S_0$, considered before.

\begin{remark} For quantum thick morphisms, there are two different kinds of power expansions: the   expansion  in Planck's constant $\hbar$ and the expansions already present for  classical thick morphisms (formal power expansions   for the pullbacks and   compositions), which can be compared with expansions ``in the coupling constant''. The source of latter are the higher order terms in  momenta in generating functions, which in particular result  in    coupled equations for determining the stationary phase points. See also   Appendix~\ref{sec.append}.
\end{remark}


\section{Quantum thick morphisms: application to homotopy algebras} \label{sec.quantumhomot}

Now we turn to application of quantum microformal morphisms to homotopy bracket structures. Since the initial motivation for introducing ``classical'' microformal morphisms was the search for a construction of $L_{\infty}$-morphisms for homotopy Poisson or Schouten brackets, it is natural to ask about the respective position of the quantum version.

For the ``quantum'' context we need to recall how a  bracket structure is generated by a   differential operator. Let $A$ be a commutative associative  superalgebra with unit. Suppose $\D$ is a linear operator acting on $A$. One can say when $\D$ is a differential operator (d.~o.) of order (less or equal to) $n$. This is defined by induction: $\D$ is of order $0$ if it commutes with multiplication by elements of $A$ and of order   $n$ if for all $a\in A$ the commutator $[\D,a]$ is of order $n-1$. (By using Hadamard's lemma, one can see that for a smooth manifold this leads to the usual definition with partial derivatives.) Such an understanding can be traced back to A.~Grothendieck~\cite[Ch.\,IV\,\S\,16.8]{groth:ega4}. J.-L.~Koszul~\cite{koszul:crochet85}   extracted from it a construction of a sequence of multilinear operations (later generalized by F.~Akman from commutative to other algebras, see.e.g.~\cite{akman:genBV}), which we shall call `brackets',\footnote{Hopefully, no confusion with the Koszul brackets  on differential forms  considered in Section~\ref{sec.algebroid}.} and which are defined as follows: for an arbitrary linear operator $\D$ on an algebra $A$ and for elements  $a_1,\ldots,a_k\in A$, where $k\geq 0$,   set
\begin{equation}\label{eq.koszbrack}
    \{a_1,\ldots,a_k\}_{\D}:=[\ldots [\D,a_1],\ldots,a_k](1)\,.
\end{equation}
For $k=0,1,2,3$ one can find
\begin{align*}
    \{\varnothing\}_{\D}&=\D(1)\,,\\
    \{a\}_{\D}&= \D(a)-\D(1)a\,,\\
    \{a,b\}_{\D}&=\D(ab)-\D(a)b-(-1)^{\at\bt}\D(b)a +\D(1)ab\,,\\
    \{a,b,c\}_{\D}&= \D(abc)-\D(ab)c-(-1)^{\bt\ct}\D(ac)  b-(-1)^{\at(\bt+\ct)}\D(bc) a \\
    & \quad \quad \quad  +\D(a)  bc+(-1)^{\at\bt}\D(b)ac+(-1)^{(\at+\bt)\ct}\D(c) ab-\D(1)abc\,,
\end{align*}
and  an expression  of this form can be written for arbitrary $k$, see below.
Koszul's construction is an example of  `higher derived brackets'~\cite{tv:higherder}.
The  brackets are symmetric  in the supersense  and have parity equal to the parity of $\D$. For any $k$, they satisfy the identity
\begin{multline}\label{eq.notleibn}
    \{a_1,\ldots,a_{k-1},a_ka_{k+1}\}_{\D} =  \{a_1,\ldots,a_{k-1},a_k\}_{\D}a_{k+1}+(-1)^{\alpha_k}  a_k\{a_1,\ldots,a_{k-1},a_{k+1}\}_{\D}\\
     +\{a_1,\ldots,a_{k-1},a_k,a_{k+1}\}_{\D}\,,
\end{multline}
where $\alpha_k =\at_k(\Dt+\at_1+\ldots+\at_{k-1})$,
which means
that
the $(k+1)^{\text{st}}$   bracket measures the failure of the $k^{\text{th}}$  bracket to be a  derivation in its arguments. If $\D$ is a differential operator of order $n$, then all brackets with more than $n$ arguments vanish, the top bracket is a multiderivation and in the formula for it there is no need to evaluate at $1$,
\begin{equation*}
    \{a_1,\ldots,a_n\}_{\D}=[\ldots[\D,a_1],\ldots,a_n]\,.
\end{equation*}
The top bracket can be identified with the \emph{principal symbol} of a differential operator. We refer to the operator $\D$ as the \emph{generating operator} of the sequence of brackets $\{-,\ldots,-\}_{\D}$.
\begin{remark} For arbitrary $k$, the expression for the $k^{\text{th}}$ bracket generated by $\D$ is
\begin{equation}\label{eq.koszbrackexpl}
    \{a_1,\ldots,a_k\}_{\D}=\sum_{s=0}^k(-1)^s\!\!\!\!\!\!\!\sum_{\text{$(k-s,s)$-shuffles}}\!\!\!\!\!\!\! \new{(-1)^{\a}}\,\D(a_{\tau(1)}\ldots { a_{\tau(k-s)})}\,a_{\tau(k-s+1)}\ldots a_{\tau(k)}\,,
\end{equation}
where \new{$(-1)^{\a}= (-1)^{\a(\tau;\at_1,\ldots,\at_k)}$}  is the standard `Koszul sign'   for  permutation of commuting factors of given parities.  (If all elements $a_1,\ldots, a_k$ are even, then $(-1)^{\a(\tau;\at_1,\ldots,\at_k)}=1$.)
\end{remark}

If $\D$ is odd, the brackets are also odd and one may ask about their Jacobiators. As shown in~\cite{tv:higherder}, the sequence of the Jacobiators is generated by the operator $\D^2=\frac{1}{2}[\D,\D]$. In particular, if $\D^2=0$, all the Jacobiators vanish and the brackets generated by $\D$ make $A$ an $L_{\infty}$-algebra (in the symmetric version).\footnote{This was first found in physics literature related with the Batalin--Vilkovisky formalism, see e.g.~\cite{bering:higher}.}

Note however that we do not obtain an $S_{\infty}$-algebra  (or `homotopy Schouten' algebra) in this way because   the Leibniz identity is not satisfied. Following~\cite{tv:higherder}, we can modify Koszul's construction to resolve this problem. Consider $A_{\hbar}:=A[[\hbar]]$. Define \emph{$\hbar$-differential operators} ($\hbar$-d.~o.'s)  as follows. Let $\D$ be a linear operator on $A_{\hbar}$. $\D$ is an $\hbar$-d.~o. of order $0$ if it commutes with the multiplication by all $a\in A_{\hbar}$. $\D$ is an $\hbar$-d.~o. of order $n$ if for all $a\in A_{\hbar}$, the operator $[\D,a]=(-i\hbar)\D'_a$, where $\D'_a$ is an $\hbar$-d.~o. of order $n-1$. For example, if $\D'$ is a d.~o. of order $n$ in the usual sense, then the operator $\D=(-i\hbar)^n\D'$ is an $\hbar$-d.~o. of order $n$.

\begin{example} \label{ex.hdo}
On a (super)manifold $M$, an arbitrary  $\hbar$-differential operator of order $n$ has the form
\begin{equation}
    \D=      (-i\hbar)^n  A^{a_1\ldots a_n}_{\hbar}(x)\,\p_{a_1}\dots\p_{a_n} +  (-i\hbar)^{n-1} A^{a_1\ldots a_{n-1}}_{\hbar}(x)\,\p_{a_1}\dots\p_{a_{n-1}}+\ldots + A^0_{\hbar}(x)\,.
\end{equation}
\end{example}

We note that the algebra of oscillatory wave functions introduced above is stable under $\hbar$-differential operators. (It is not stable under arbitrary differential operators because they may create the factors of $\hbar^{-1}$.)

For an  $\hbar$-d.~o. $\D$ of arbitrary order, all $k$-fold commutators $[\ldots [\D,a_1],\ldots,a_k]$ are divisible by $(-i\hbar)^k$, and we can (re)define the brackets generated by $\D$ by setting:
\begin{equation}\label{eq.hbrack}
    \{a_1,\ldots,a_k\}_{\D,\hbar}:=(-i\hbar)^{-k}[\ldots [\D,a_1],\ldots,a_k](1)\,.
\end{equation}
We can also introduce the corresponding ``classical'' brackets by
\begin{equation}\label{eq.classbrack}
    \{a_1,\ldots,a_k\}_{\D,0}:=\lim_{\hbar\to 0}\;(-i\hbar)^{-k}[\ldots [\D,a_1],\ldots,a_k](1)
\end{equation}
(the limit has the meaning of substituting $0$ in the nonnegative power expansion in $\hbar$). We   refer to~\eqref{eq.hbrack} as the ``quantum'' brackets as opposed to the ``classical'' brackets~\eqref{eq.classbrack}.
The quantum brackets satisfy the identity
\begin{multline}\label{eq.hleibn}
    \{a_1,\ldots,a_{k-1},a_ka_{k+1}\}_{\D,\hbar} =\{a_1,\ldots,a_{k-1},a_k\}_{\D,\hbar}\,a_{k+1}+\new{(-1)^{\att} a_k\{a_1,\ldots,a_{k-1},a_{k+1}\}_{\D,\hbar}}\\
     +(-i\hbar)\{a_1,\ldots,a_{k-1},a_k,a_{k+1}\}_{\D,\hbar}\,,
\end{multline}
which for $\hbar\to 0$ becomes the derivation property
\begin{equation}\label{eq.classleibn}
    \{a_1,\ldots,a_{k-1},a_ka_{k+1}\}_{\D,0} =\{a_1,\ldots,a_{k-1},a_k\}_{\D,0}\,a_{k+1}+(-1)^{\att}   a_k\{a_1,\ldots,a_{k-1},a_{k+1}\}_{\D,0}\,.
\end{equation}
Here   $\att=  {\at_k(\Dt+\at_1+\ldots+\at_{k-1})}$.
We call the sequence of all classical brackets generated by $\D$, the \emph{principal symbol} of an $\hbar$-differential operator $\D$. On a (super)manifold, since the classical brackets are symmetric multiderivations of the algebra of functions, the   principal symbol can be identified with an inhomogeneous polynomial in momentum variables.
(In the language of Section~\ref{sec.algebroid}, it is the  `master Hamiltonian' of the   brackets.)

\begin{example} For an $\hbar$-differential operator of Example~\ref{ex.hdo}, the principal symbol is
\begin{equation}
    H(x,p)=
        A^{a_1\ldots a_n}_{0}(x)\,p_{a_1}\dots p_{a_n} +    A^{a_1\ldots a_{n-1}}_{0}(x)\,p_{a_1}\dots p_{a_{n-1}}+\ldots + A^0_{0}(x)\,.
\end{equation}
which is an inhomogeneous  fiberwise-polynomial function on $T^*M$, well-defined independently of a choice of coordinates! (Subscript $0$ means substituting $0$ for $\hbar$ in the coefficients.)
\end{example}

\begin{remark} If we only keep the condition that all $k$-fold commutators $[\ldots [\D,a_1],\ldots,a_k]$ be divisible by $(-i\hbar)^k$, formula~\eqref{eq.hbrack} still makes sense and we obtain, generally, an infinite sequence  of brackets. We shall refer to such operators as \emph{formal $\hbar$-differential operators}. On manifolds, this gives operators whose principal symbols are formal power series in momenta. Algebraic constructions  here agree with the  known notion of  $\hbar$-pseudodifferential operators   defined in local coordinates by integrals
\begin{equation*}
    (\D u)(x) = \iint \Dbar p Dx'  e^{\frac{i}{\hbar}(x^a-{x'}^{a})p_a} H_{\hbar}(x,p) u(x')\,,
\end{equation*}
with a function $H_{\hbar}(x,p)$  from a suitable symbol class (see e.g.~\cite{shubin:book}).     
Here the  ``full symbol''  $H_{\hbar}(x,p)$   is coordinate-dependent, but the principal symbol $H(x,p)=H_{0}(x,p)$ is   well defined as a function on $T^*M$. 
\end{remark}

Suppose an odd operator $\D$ squares to $0$. Consider   the quantum brackets~\eqref{eq.hbrack}. They define  an $L_{\infty}$-algebra (in the odd symmetric version) and additionally satisfy the modified Leibniz identity~\eqref{eq.hleibn}.  We shall call such an algebraic structure an   \emph{$S_{\infty,\hbar}$-algebra}.   (So that for $\hbar=0$ we come back to an $S_{\infty}$-algebra, $S_{\infty,0}=S_{\infty}$.)  We shall give a formula for the corresponding homological vector field, as well as a formula for the master Hamiltonian for the classical $S_{\infty}$-algebra (i.e. the principal symbol of $\D$).

\begin{lemma}\label{lem.qa}
The quantum  brackets~\eqref{eq.hbrack} correspond to a formal vector field $Q$ on an algebra $A$ (more accurately,  on the corresponding  supermanifold $\mathbf{A}$), where
\begin{equation}\label{eq.qforhbrack}
    Q= e^{-\frac{i}{\hbar}a}\D\bigl(e^{\frac{i}{\hbar}a}\bigr)\,\var{}{a}\,.
\end{equation}
Here $a\in A$ and $\D(e^{\frac{i}{\hbar}a})$ denotes the application of the operator  to the function.
\end{lemma}
\begin{proof} The formal vector field corresponding to a sequence of symmetric multilinear functions of fixed parity on a superspace $A$ is the formal sum
\begin{equation*}
    Q(a)=\sum_{k=0}^{+\infty} \frac{1}{k!}\underbrace{\{a,\ldots,a\}}_{k}\,,
\end{equation*}
see e.g.~\cite{tv:higherder}. Here $a\in A$ is a `running' even element (or a point of the corresponding supermanifold). A vector field here is identified with a vector-function. It can be expressed also as
\begin{equation*}
    Q=\sum_{k=0}^{+\infty} \frac{1}{k!}\underbrace{\{a,\ldots,a\}}_{k}\,\var{}{a}\,,
\end{equation*}
meaning an infinitesimal shift $a\mapsto a +\e Q(a)$\,. The relation of the vector field $Q$ with the given multilinear functions is by the higher derived bracket formula~\cite{tv:higherder}
\begin{equation*}
    \{a_1,\ldots,a_k\}=[\ldots[Q,a_1],\ldots,a_k](0)
\end{equation*}
(the value of a vector field at the origin). Here vectors $a_i$ are regarded as constant vector fields. Now, to obtain~\eqref{eq.qforhbrack}, we take an even element $a$ in the algebra $A$ and consider
\begin{equation*}
     \underbrace{\{a,\ldots,a\}}_{k}{}_{\D,\hbar} =
    \bigl[\ldots \bigl[\D,\frac{i}{\hbar} a\bigr],\ldots,\frac{i}{\hbar} a\bigr](1)=
    \left(\left(\ad\bigl(-\frac{i}{\hbar}  a\bigr)\right)^k \D\right)(1)\,,
\end{equation*}
hence
\begin{multline*}
    Q(a)=\sum_{k=0}^{+\infty} \frac{1}{k!}\left(\left(\ad\bigl(-\frac{i}{\hbar}  a\bigr)\right)^k \D\right)(1)=
   \left( e^{\ad\bigl(-\frac{i}{\hbar}  a\bigr)}\D\right)(1)=
   \left(\Ad(e^{-\frac{i}{\hbar}  a})\D\right)(1)=\\
   \left(e^{-\frac{i}{\hbar}  a} \D e^{\frac{i}{\hbar}  a} \right)(1)=
   e^{-\frac{i}{\hbar}  a} \D \bigl(e^{\frac{i}{\hbar}  a} (1)\bigr)=
   e^{-\frac{i}{\hbar}  a} \D \bigl(e^{\frac{i}{\hbar}  a}\bigr)\,.
\end{multline*}
\end{proof}

\begin{lemma} In the differential-geometric setting, the principal symbol of $\D$  or the master Hamiltonian of the classical brackets~\eqref{eq.classbrack}   is given by
\begin{equation}\label{eq.hforclassbrack}
    H(x,p)=\lim_{\hbar\to 0} e^{-\frac{i}{\hbar}x^ap_a}\D(e^{\frac{i}{\hbar}x^ap_a})
\end{equation}
\end{lemma}
\begin{proof} Recall that the master Hamiltonian $H$ of symmetric brackets is defined by the relation~\cite{tv:higherder}
\begin{equation*}
    \{f_1,\ldots,f_k\}=(\ldots (H,f_1),\ldots,f_k){\left|\vphantom{\sum}\right.}_{M}\,,
\end{equation*}
for functions $f_i\in\fun(M)$\,. Hence, in local coordinates,
\begin{equation*}
    H(x,p)= \sum_{k=0}^{+\infty} \frac{1}{k!} \{x^{a_1}p_{a_1},\ldots,x^{a_k}p_{a_k}\}\,,
\end{equation*}
where the momentum variables $p_a$ are treated as parameters when the brackets are taken; therefore, for the classical brackets generated by an operator $\D$, we obtain
\begin{equation*}
    H(x,p)=\lim_{\hbar\to 0}\sum_{k=0}^{+\infty}  \frac{1}{k!}\bigl[\ldots \bigl[\D, \frac{i}{\hbar}x^{a_1}p_{a_1}\bigr],\ldots,\frac{i}{\hbar}x^{a_k}p_{a_k}\bigr](1)=
    \lim_{\hbar\to 0}e^{-\frac{i}{\hbar} x^{a}p_{a}} \D \bigl(e^{\frac{i}{\hbar}  x^{a}p_{a}}\bigr)
\end{equation*}
(where we have effectively repeated the argument used in the proof of Lemma~\ref{lem.qa}).
\end{proof}

\begin{remark} For neither formula~\eqref{eq.qforhbrack} nor~\eqref{eq.hforclassbrack}  it is important that the operator $\D$ generating the brackets is odd or satisfies $\D^2=0$. In particular, it makes sense to consider ``$L_{\infty}$-morphisms'' of the infinite sequences of brackets generated by arbitrary  operators $\D$ without such assumptions. It is interesting what such morphisms would mean in the classical context of partial differential equations.
\end{remark}

Now we shall give a ``quantum analog'' of Theorem~6 of~\cite{tv:nonlinearpullback} that says that Poisson thick morphisms induce $L_{\infty}$-morphisms of homotopy Poisson brackets.

\begin{definition}
We say that $M$ is a \emph{Batalin-Vilkovisky manifold} or shortly a \emph{BV-manifold} if $M$ is a supermanifold equipped with an odd formal $\hbar$-differential operator  $\D$   of square zero. The operator $\D$ is referred to as the \emph{BV-operator}. If $M_1$ and $M_2$ are  BV-manifolds with the BV-operators $\D_1$ and $\D_2$, we    say that a quantum thick morphism $\hat\F\co M_1\ttoq M_2$ is a \emph{quantum BV-morphism}  if
\begin{equation}\label{eq.qbvmor}
    \D_1\circ \hat \F^* = \hat\F^*\circ \D_2\,.
\end{equation}
\end{definition}

The BV-operator on a BV-manifold  $M$ specifies an $S_{\infty,\hbar}$-structure  on the algebra  $\funh(M_1)$.
We shall show that a quantum BV-morphism  $\hat\F\co M_1\ttoq M_2$  induces an $L_{\infty}$-morphism of the corresponding $S_{\infty,\hbar}$-algebras. Note that it cannot be the pullback $\hat \F^*$ itself, since $\hat \F^*$ is linear and we are looking for a nonlinear map of the  function supermanifolds
\begin{equation*}
    \funnh(M_2)\to \funnh(M_1)\,.
\end{equation*}
For a quantum thick morphism $\hat \F$ (not necessarily BV), define $\hat \F^!$ by
\begin{equation}\label{eq.phishriek}
    \hat \F^!(g):=\frac{\hbar}{i}\,\ln\left(\hat\Phi^* e^{\frac{i}{\hbar}g}\right)\,,
\end{equation}
for a $g\in \funnh(M_2)$\,. If we denote $\exph g :=\exp (\frac{i}{\hbar}g)$, $\lnh  f := \frac{\hbar}{i}\,\ln f$, then
\begin{equation}\label{eq.phishriek2}
    \hat \F^! =\lnh\circ \, \hat\Phi^*\!\circ \exph\,.
\end{equation}
For the composition of quantum thick morphisms,
\begin{equation*}
  M_1\tttoq{\hat\Phi_{21}} M_2\tttoq{\hat\Phi_{32}} M_3\,,
 \end{equation*}
we have
\begin{equation*}
    (\hat \F_{32}\circ \hat \F_{21})^!=\hat \F_{21}^!\circ \hat \F_{32}^!\,.
\end{equation*}
\begin{theorem} \label{thm.qbvaslinf}
If $\hat\F\co M_1\ttoq M_2$ is a quantum BV-morphism, then $\hat \F^!$ is an $L_{\infty}$-morphism of the  $S_{\infty,\hbar}$-algebras of functions. In greater detail, the map
\begin{equation*}
    \hat \F^!\co \funnh(M_2)\to \funnh(M_1)
\end{equation*}
is a morphism of $Q$-manifolds, where the homological vector fields $Q_i\in\Vect\left(\funnh(M_i)\right)$    corresponding to the BV-operators $\D_i$, $i=1,2$,  are given by Lemma~\ref{lem.qa}.
\end{theorem}
\begin{proof} By Lemma~\ref{lem.qa}, the homological vector fields $Q_i$ regarded as infinitesimal shifts on the supermanifold $\funnh(M_i)$ are given by
\begin{equation*}
    Q_i(f)=e^{-\frac{i}{\hbar}f}\D_i\bigl(e^{\frac{i}{\hbar}f}\bigr)\,,
\end{equation*}
so that $f\mapsto f+\e Q_i(f)$\,. We need to show that $\hat \F^!$ commutes with these shifts. Indeed, let $g\in \funnh(M_2)$;  apply the infinitesimal shift by $Q_2$ followed by $\hat \F^!$. We obtain
\begin{multline*}
    \hat \F^!(g+\e Q_2(g))= \hat \F^!(g+\e e^{-\frac{i}{\hbar}g}\D_2(e^{\frac{i}{\hbar}g}))=
    \frac{\hbar}{i}\ln \hat \F^* \exp\left(\frac{i}{\hbar}\left(g+\e e^{-\frac{i}{\hbar}g}\D_2(e^{\frac{i}{\hbar}g})\right)\right)=\\
    \frac{\hbar}{i}\ln \hat \F^* \left(e^{\frac{i}{\hbar}g}\Bigl(1+ \e \,\frac{i}{\hbar}\, e^{-\frac{i}{\hbar}g}\D_2(e^{\frac{i}{\hbar}g})\Bigr)\right)=
    \frac{\hbar}{i}\ln \hat \F^* \left(e^{\frac{i}{\hbar}g}+ \e \,\frac{i}{\hbar}\,\D_2(e^{\frac{i}{\hbar}g})\right)=\\
    \frac{\hbar}{i}\ln  \left(\hat \F^*e^{\frac{i}{\hbar}g}+ \e \,\frac{i}{\hbar}\,\hat \F^*\bigl(\D_2(e^{\frac{i}{\hbar}g})\bigr)\right)=
    \frac{\hbar}{i}\ln  \left(\hat \F^*e^{\frac{i}{\hbar}g}+ \e \,\frac{i}{\hbar}\,\D_1\bigl(\hat \F^* e^{\frac{i}{\hbar}g}\bigr)\right)=\\
    \frac{\hbar}{i}\ln \hat \F^*e^{\frac{i}{\hbar}g} + \e   (\hat\F^*e^{\frac{i}{\hbar}g})^{-1} \D_1\bigl(\hat \F^* e^{\frac{i}{\hbar}g}\bigr)=
    \hat \F^! (g)+ \e   (\hat\F^*e^{\frac{i}{\hbar}g})^{-1} \D_1\bigl(\hat \F^* e^{\frac{i}{\hbar}g}\bigr)\,.
 \end{multline*}
Here we used the commutativity condition~\eqref{eq.qbvmor}. Now apply first $\hat \F^!$, then the infinitesimal shift by $Q_1$. We have
\begin{equation*}
    \hat \F^!(g) +\e Q_1\bigl(\hat \F^!(g)\bigr)=  \hat \F^!(g) +\e  e^{-\frac{i}{\hbar}\hat \F^!(g)}\D_1(e^{\frac{i}{\hbar}\hat \F^!(g)})\,;
\end{equation*}
note that
\begin{equation*}
    e^{\frac{i}{\hbar}\hat \F^!(g)}= \hat \F^* e^{\frac{i}{\hbar}g}\,.
\end{equation*}
Hence
\begin{equation*}
    \hat \F^!(g) +\e Q_1\bigl(\hat \F^!(g)\bigr)=  \hat \F^!(g) +\e  (\hat \F^* e^{\frac{i}{\hbar}g})^{-1}\D_1(\hat \F^* e^{\frac{i}{\hbar}g})\,,
\end{equation*}
which is exactly as above. Thus, $\hat \F^!$ intertwines $Q_1$ and $Q_2$ as claimed.
\end{proof}

\begin{remark}
The definition of the map $\hat\F^{!}$ by formulas \eqref{eq.phishriek}, \eqref{eq.phishriek2} is motivated by the stationary phase method (\emph{before} the limit $\hbar\to 0$ is taken). By Theorem~\ref{thm.quantstatph}, we have  $\lim\limits_{\hbar\to 0}\hat \F^{!}=\F^*$, where $\F$ is the classical thick morphism corresponding to a quantum thick morphism $\hat\F$\,. On the other hand, this construction is entirely algebraic and makes sense, together with an analog of Theorem~\ref{thm.qbvaslinf}, in the abstract setting as follows.
\end{remark}

Call an odd formal $\hbar$-differential operator $\D$ on a commutative unital superalgebra $A$ satisfying $\D^2=0$, a \emph{BV-operator}. Call an algebra $A$ endowed with such a $\D$, a \emph{BV-algebra}. (This is not standard terminology, but convenient for our present purpose.) Every BV-operator generates an infinite sequence of brackets by~\eqref{eq.hbrack}, which defines an $S_{\infty,\hbar}$-structure on the algebra  $A$. In fact, an  $S_{\infty,\hbar}$-structure is completely determined by its $0$- and $1$-brackets, as all the higher brackets are inductively obtained as the discrepancies in the Leibniz identities. Since we can   recover $\D$ as $\D(a)=\frac{\hbar}{i}\{a\}_{\D,\hbar}+\{\varnothing\}_{\D,\hbar}a$ and the Jacobi identities will give $\D^2=0$, the notions of a BV-algebra and $S_{\infty,\hbar}$-algebra coincide. Note also that since the parameter $\hbar$ plays a formal role here, we can set $-i\hbar\equiv 1$; then being a `formal $\hbar$-differential operator' becomes empty condition and the brackets~\eqref{eq.hbrack} turn back to the original operations introduced by Koszul. Let $A_1$ and $A_2$ be BV-algebras and let $\Phi\co A_1\to A_2$ be an even linear transformation  such that $\F\circ \D_1=\D_2\circ \F$. ($\F$ is not assumed to be a homomorphism with respect to the associative multiplication.) Call such a $\F$, a \emph{BV-morphism}. Define
\begin{equation*}
    \F_{!}:=\lnh \circ \,\F \circ \exph\,.
\end{equation*}
(There is an obvious functoriality relation $(\F_{1}\circ\F_{2})_{!}= {\F_{1}}_{!} \circ {\F_{2}}_{!}$\,. If $\F$ is a homomorphism, then $\F_{!}=\F$.)

\begin{theorem}
If $\F$ is a BV-morphism, then $\F_{!}$ is an $L_{\infty}$-morphism of the $S_{\infty,\hbar}$-structures.
\end{theorem}
\begin{proof}The proof of Theorem~\ref{thm.qbvaslinf} applies verbatim.
\end{proof}

\begin{corollary}[from Theorem~\ref{thm.qbvaslinf}] If $\hat\F\co M_1\ttoq M_2$ is a BV-morphism, the classical pullback $\F^*=\lim\limits_{\hbar\to 0}\hat\F^{!}$ is an $L_{\infty}$-morphism of the classical $S_{\infty}$-structures.
\end{corollary}
\begin{proof} Indeed, $\hat\F^{!}$ is an $L_{\infty}$-morphism of the   $S_{\infty,\hbar}$-structures. Passing to the limit $\hbar\to 0$ gives the claim.
\end{proof}

In~\cite{tv:nonlinearpullback}, we showed, for $S_{\infty}$-manifolds, that if a thick morphism $\F\co M_1\tto M_2$ is Poisson, i.e., the master Hamiltonians on $M_1$ and $M_2$ are $\F$-related, then the pullback $\F^*$ is an $L_{\infty}$-morphism of the homotopy Schouten brackets. We shall now relate  this with Theorem~\ref{thm.qbvaslinf}.

\begin{theorem} \label{thm.bvpoisson}
Let $M_1$ and $M_2$  be BV-manifolds and let $\hat\F\co M_1\ttoq M_2$ be a quantum BV-morphism. Then its classical limit $\F\co M_1\tto M_2$ is a Poisson morphism for the induced $S_{\infty}$-structures.
\end{theorem}
\begin{proof} Let $H_i\in\fun(T^*M_i)$, $i=1,2$, be the master Hamiltonians for the $S_{\infty}$-structures on $M_1$ and $M_2$ arising from the BV-operators $\D_1$ and $\D_2$. In other words, $H_1$ and $H_2$ are the principal symbols of    $\D_1$ and $\D_2$. We need to show that $H_1$ and $H_2$ are  $\F$-related, i.e., $\pi_1^*H_1=\pi_2^*H_2$ on the canonical relation $\F\subset M_2\times (-M_1)$~\cite{tv:nonlinearpullback}. This is a Hamilton--Jacobi equation
\begin{equation}\label{eq.poissonhj}
    H_1\Bigl(x,\der{S}{x}\,\Bigr)=H_2\Bigl(\!(-1)^{\qt}\der{S}{q},q\Bigr)\,,
\end{equation}
where $S(x,q)$ is the generating function of $\F$\,.
We are given that
\begin{equation} \label{eq.bv}
    \D_1\circ \hat \F^* = \hat\F^*\circ \D_2\,.
\end{equation}
In order to deduce~\eqref{eq.poissonhj} from~\eqref{eq.bv}, write  $\D_1$ and $\D_2$ as integral operators:
\begin{equation*}
    (\D_1u)(x)=\int Dx'\Dbar p'\,e^{\frac{i}{\hbar}(x-x')p'} H_{1,\hbar}(x,p')\,u(x')
\end{equation*}
and
\begin{equation*}
    (\D_2w)(y)=\int Dy'\Dbar q'\,e^{\frac{i}{\hbar}(y-y')q'} H_{2,\hbar}(y,q')\,w(y')\,.
\end{equation*}
Here $H_{1,\hbar}$ and $H_{2,\hbar}$ are   full symbols, which are coordinate-dependent objects. When $\hbar\to 0$, we get from them   the principal symbols $H_1=H_{1,0}$ and $H_2=H_{2,0}$, which    we need, and they are well-defined functions on $T^*M_1$ and $T^*M_2$. We have
\begin{equation*}
    (\D_1\hat\F^*w)(x)=\int Dx'\Dbar p'\,e^{\frac{i}{\hbar}(x-x')p'} H_{1,\hbar}(x,p')
    \int Dy \Dbar q \,e^{\frac{i}{\hbar}\left(S_{\hbar}(x',q)-yq\right)}w(y)
\end{equation*}
and
\begin{equation*}
    (\hat\F^*\D_2w)(x)=\int Dy \Dbar q \,e^{\frac{i}{\hbar}\left(S_{\hbar}(x,q)-yq\right)}
    \int Dy'\Dbar q'\,e^{\frac{i}{\hbar}(y-y')q'} H_{2,\hbar}(y,q')\,w(y')\,,
\end{equation*}
where $S_{\hbar}(x,q)$ is the quantum generating function for $\hat \F$. Take $w=e^{\frac{i}{\hbar}yc}$ as a `test function' as in Example~\ref{ex.testfunction} and obtain, respectively,
\begin{equation*}
    (\D_1\hat\F^*w)(x)=\int\! Dx'\Dbar p' Dy\, \Dbar q\;
    e^{\frac{i}{\hbar}\left(S_{\hbar}(x',q) + (x-x')p' + y(c-q)\right)}\,
    H_{1,\hbar}(x,p')
\end{equation*}
and
\begin{equation*}
    (\hat\F^*\D_2w)(x)=\int\! Dy\, \Dbar q \,Dy'\Dbar q'\,
    e^{\frac{i}{\hbar}\left(S_{\hbar}(x,q) + y(q'-q)   +y'(c-q')\right)}\,
    H_{2,\hbar}(y,q')\,.
\end{equation*}
In each case, the integral is simplified by the integration giving a delta-function and the subsequent integration with the delta-function. This gives finally
\begin{equation*}
    (\D_1\hat\F^*w)(x)=\int\! Dx'\Dbar p'\;
    e^{\frac{i}{\hbar}\left(S_{\hbar}(x',c) + (x-x')p'\right)}\,
    H_{1,\hbar}(x,p')
\end{equation*}
and
\begin{equation*}
    (\hat\F^*\D_2w)(x)=\int\! Dy\, \Dbar q \,
    e^{\frac{i}{\hbar}\left(S_{\hbar}(x,q) + y(c-q)\right)}\,
    H_{2,\hbar}(y,c)\,.
\end{equation*}
Now we apply the stationary phase method.
The stationary points for the phases
are specified, respectively, by the equations
\begin{equation*}
    x^a-{x'}^a=0\,, \quad \der{S_{\hbar}}{x^a}(x',c)-p'_a=0
\end{equation*}
and
\begin{equation*}
    c_i-q_i=0\,, \quad \der{S_{\hbar}}{q_i}(x,q)-(-1)^{\itt}y^i=0\,,
\end{equation*}
and both Hessians are equal to $1$.  Altogether
we obtain
\begin{equation*}
    (\D_1\hat\F^*w)(x)=
    e^{\frac{i}{\hbar} S_{\hbar}(x,c)}\,
    H_{1,\hbar}\Bigl(\!x,\der{S_{\hbar}}{x}(x,c)\!\Bigr)\,\bigl(1+ O(\hbar)\bigr)
\end{equation*}
and
\begin{equation*}
    (\hat\F^*\D_2w)(x)=
    e^{\frac{i}{\hbar} S_{\hbar}(x,c)}\,
    H_{2,\hbar}\Bigl(\!(-1)^{\qt}\der{S_{\hbar}}{q}(x,c),c\Bigr)
    \,\bigl(1+ O(\hbar)\bigr)
    \,.
\end{equation*}
The phase factors coincide; so by   eliminating them and setting  $\hbar\to 0$, we arrive at the equality
\begin{equation*}
    H_{1}\Bigl(\!x,\der{S_{0}}{x}(x,c)\!\Bigr)= H_{2}\Bigl(\!(-1)^{\qt}\der{S_{0}}{q}(x,c),c\Bigr)\,,
\end{equation*}
as desired because $S_0=S$ is the generating function of $\F$.
\end{proof}

Theorem~\ref{thm.bvpoisson} is similar with Egorov's fundamental theorem about canonical transformations of pseudodifferential operators~\cite{egorov:canonicpdo1969,egorov:canonicpdo1971}, see also Fedoryuk~\cite{fedoryuk:pdo1971}, which was one of the chief early sources for the theory of Fourier integral operators~\cite{hoer:fio1-1971}.)
More precisely, in Egorov's theorem, Fourier integral operators are constructed that intertwine pseudodifferential operators whose principal symbols are related by a canonical transformation.  The statement of our Theorem~\ref{thm.bvpoisson} is analogous to the inverse Egorov theorem. An analog of the direct Egorov theorem would be the following statement that should also be true:  every   $S_{\infty}$-structure, i.e.,  homotopy Schouten brackets for an arbitrary manifold, can be lifted to an  $S_{\infty,\hbar}$-structure or equivalently to a BV-operator $\D$, and every Poisson thick morphism between  $S_{\infty}$-manifolds can be lifted to a quantum  BV-morphism, which intertwines  $\D_1$ and $\D_2$.

\section*{Conclusions and discussion}

Let us summarize  what we have achieved so far.
We have introduced a new class of morphisms between smooth manifolds  (or supermanifolds).
They include  smooth maps, but are not themselves maps in the ordinary sense, i.e., not maps of sets.  In practice they are described  by their ``generating functions'' $S(x_1,p_2)$ depending as arguments on position variables on the source manifold and  momentum variables on the target manifold. Geometric objects underlying such morphisms (which we called ``thick'' or ``microformal'') are canonical relations between the cotangent bundles of the source and target, of a particular type maximally close to those induced by ordinary maps of the base manifolds. Namely, the relations that project without degeneration onto the source manifold and onto the fibers of the cotangent of the target; for the latter condition to make invariant sense, we are forced to consider our relations as formal. Hence the generating functions are formal power expansions in the cotangent directions. This explains the terminology `microformal morphisms' and `microformal geometry'. Since generating functions that differ by a constant define the same canonical relation and we actually need the functions themselves, not up to  constants,   we may think that we work with ``framed'' relations (meaning a  choice of additive constants).

The composition  of thick morphisms between (super)manifolds is of course the standard composition of relations; however, the statement is that the resulting relation is of the same type and that the  composition law is itself formal. The generating function of the composition of two thick morphisms is expressed as  a formal power expansion in their generating functions. This composition law is local (depends only on the values of the generating functions and their derivatives of  orders bounded from above in each term of the expansion). It is obtained by an iterative procedure. A similar iterative procedure defines the action of a thick morphism  on smooth functions, i.e., the pullback.   A distinctive feature of the pullback is that it is in general   a  nonlinear transformation.

This nonlinearity, first of all, forces  us  to distinguish between even functions and odd functions. There are two parallel constructions, of `even' and `odd' thick morphisms, corresponding to these two cases. Secondly, since the pullback of functions by a thick morphism of supermanifolds is in general nonlinear and in particular non-additive, it cannot be a ring homomorphism in the ordinary sense. This at the first glance contradicts the   philosophy of ``space-algebra duality''  according to which to ``spaces'' there correspond algebras (interpreted as algebras of functions) and to  maps of spaces  there correspond algebra homomorphisms  (with reversed direction).
However, it turns out that the   derivative  of the pullback  by a thick morphism, which is automatically a linear transformation, is the pullback  in the ordinary sense (by some perturbed map  between the source and target) and hence   is  an algebra homomorphism. This naturally suggests a ``nonlinear generalization'' of the notion of
algebra homomorphisms  and the corresponding generalization   of the algebra/geometry duality.    Such a generalization is yet to be explored. The author wishes to stress that his initial motivation was very concrete, namely, to obtain a method of construction of $L_{\infty}$-morphisms for homotopy Poisson structures\footnote{We   mean both homotopy Poisson and homotopy Schouten structures, i.e., the strongly homotopy versions   of even and odd Poisson brackets.} and that microformal geometry is indeed successful for that and other applications, such as to vector bundles and Lie algebroids.

Still, since we have obtained two new (formal) categories, in the   versions adapted to   even functions and to odd functions,  whose objects are supermanifolds, this inevitably  leads to  further questions. Such are, in particular: (1) extending the functoriality from functions to  other geometric objects such as e.g. differential forms; (2) if the previous is successful, obtaining, further, an action of thick morphisms (possibly nonlinear) on various cohomology spaces or, e.g.  on the Fukaya categories of the cotangent bundles; (3) by making use of  these larger classes of morphisms, to explore what, e.g., group objects in the ``thick'' sense would be, and what could be obtained by gluing by thick diffeomorphisms.

There are also other specific questions which can be addressed  in future studies.
For instance, is it possible to obtain a more efficient description of the power
expansions which specify the pullback and the composition; perhaps, by some graphic
calculus.

In   microformal geometry, particularly, in applications to homotopy Poisson structures, arises prominently the Hamilton--Jacobi equation: for example, in the form of  the infinitesimal action on functions   by thick diffeomorphisms~\cite{tv:oscil},
 \begin{equation*}
    f(x)\mapsto f(x)+\e H\bigl(x,\der{f}{x}\bigr)\,,
 \end{equation*}
also as an expression of the condition for a thick morphism between homotopy Poisson manifolds to be (homotopy) Poisson, and in the formula for à homological vector field on the space of functions~\cite{tv:nonlinearpullback}. Such prominence of the Hamilton--Jacobi equation in our constructions  together with its fundamental relation with the Schr\"{o}dinger equation in quantum mechanics, has led us to  building  the quantum version of microformal geometry. In it,   nonlinear pullbacks by ``classical'' thick morphisms are replaced by  Fourier integral operators of some special kind  (resembling the early version of such  operators studied by Fock, Vishik--\`{E}skin, Fedoryuk and Egorov in 1950s-1960s).  The ``classical'' thick morphisms (in the bosonic case) are   recovered from ``quantum''  in the limit  $\hbar\to 0$.
This may be seen in hindsight as an elucidation of the classical picture. Since the first motivation for microformal morphisms was to homotopy Poisson structures and their $L_{\infty}$-morphisms, it was natural to ask about a similar application of quantum thick morphisms. This has turned out to be indeed possible by replacing   master Hamiltonians by    Batalin--Vilkovisky type $\Delta$-operators (compare~\cite{tv:laplace1,tv:laplace2}).  We see here a fascinating interplay between homotopy algebras and some purely algebraic ideas on one hand with   very classical ideas  from partial differential equations, pseudodifferential operators and Fourier integral operators, on the other. Obviously, as well as in the classical version, there are plenty of questions for further study.

\appendix
\section{A version of the stationary phase formula}\label{sec.append}

Here we recall a general type stationary phase formula and give its particular version adapted for application to quantum thick morphisms considered in Sections~\ref{sec.quantum} and~\ref{sec.quantumhomot}. We are basically following Fedoryuk's approach~\cite{fedoryuk:pdo1971}, with some modification  and simplification (and    extending it to the super case). For a general type formula, we consider an integral of the form
\begin{equation}\label{eq.statphaseint}
    I_{\phi}(a)=\!\!\int\limits_{\,\,\,\,\R{n|2m}}\!\!\!\! Dx\;e^{\frac{i}{\hbar}\phi(x)} a(x)\,.
\end{equation}
Here $\hbar$ is a formal parameter and both functions $\phi(x)$ (called ``phase'') and $a(x)$ are assumed to be formal power series in $\hbar$ over nonnegative powers. For the simplicity of notation, this dependence on $\hbar$ is not indicated explicitly. It is assumed that $a(x)$ is compactly supported and the phase $\phi(x)$ has one stationary point on the support of $a(x)$. (Obviously a more general case is reduced to this one by using partitions of unity.) Denote this point by $x_0$. There is an expansion
\begin{equation}
    \phi(x)=\phi(x_0)+\frac{1}{2}d^2\phi(x_0)(x-x_0)+\phi^{+}(x;x_0)\,,
\end{equation}
where the function $\phi^{+}(x;x_0)$ has a zero of order $3$ at $x=x_0$. Assume that the quadratic form $d^2\phi(x_0)$ is nondegenerate (that is why we need the dimension $n|2m$). We rewrite the integral as
\begin{equation}\label{eq.statphaseintexpan}
    I_{\phi}(a)=e^{\frac{i}{\hbar}\phi(x_0)}\!\!\int\limits_{\,\,\,\,\R{n|2m}}\!\!\!\! Dx\;
    e^{\frac{i}{\hbar}\frac{1}{2}d^2\phi(x_0)(x-x_0)} \, \bigl(e^{\frac{i}{\hbar}\phi^{+}(x;x_0)} a(x)\bigr)\,,
\end{equation}
which, apart from the   factor, has the general form of
\begin{equation*}
    \int 
    \!   Dx\;
    e^{\frac{i}{\hbar}\frac{1}{2}Q(x-x_0)} \, u(x)\,,
\end{equation*}
where $Q(x-x_0)$ is a nondegenerate quadratic form and $u(x)$ some `test function'. (We    suppress  the domains of integration when convenient.) Any such integral can be expressed as an application of a (formal) differential operator, as follows. For an arbitrary function or a distribution $f(x)$, an equality holds:
\begin{equation}
    \int\!  Dx\, f(x_0-x) u(x) = \tilde f \Bigl(\frac{\hbar}{i}\der{}{x}\Bigr) u(x){\left|\vphantom{\int}\right.}_{x=x_0}\,,
\end{equation}
where $\tilde f(p)$ denotes the $\hbar$-Fourier transform of $f(x)$\,. Indeed, $f(x-x')$ and $\tilde f(p)$ are respectively the kernel and full symbol of a translationally-invariant operator. In particular, for a Gaussian oscillating exponential $E(x)=e^{\frac{i}{\hbar}\frac{1}{2}Q(x)}$ on $\R{n|2m}$, its $\hbar$-Fourier transform is
\begin{equation}
    \tilde E(p)= c_{n|2m,\hbar}\, \frac{e^{\frac{i\pi}{4}\sgn Q}}{\sqrt{|\Ber Q|}} \;e^{-\frac{i}{\hbar}\frac{1}{2}Q^{-1}(p)}\,,
\end{equation}
where $c_{n|2m,\hbar}=(2\pi \hbar)^{n/2}(i\hbar)^{-m}$. Here
we use   $Q$    both for the quadratic form $Q(x)=x^ax^bQ_{ba}$ and for its matrix $Q_{ab}$,   and   $\sgn Q$ is the signature (the difference of the numbers of positive and negative squares of the even variables in the canonical form).  By $Q^{-1}(p)=Q^{ab}p_bp_a$ we denote the induced quadratic form on the momentum space, where $(Q^{ab})=(Q_{ab})^{-1}$.   It is the superanalog of the familiar formula and can be obtained by a manipulation with   standard Gaussian integrals. Hence, for any function $u(x)$, we have
\begin{equation}
    \int\limits_{\,\,\,\,\R{n|2m}}\!\!\!
    Dx\;
    e^{\frac{i}{\hbar}\frac{1}{2}Q(x-x_0)} \, u(x)=
   c_{n|2m,\hbar}\, \frac{e^{\frac{i\pi}{4}\sgn Q}}{\sqrt{|\Ber Q|}} \;
   e^{-\frac{\hbar}{i}\frac{1}{2}Q^{-1}\bigl(\der{}{x}\bigr)}
    u(x) {\left|\vphantom{\int}\right.}_{x=x_0}\,.
\end{equation}
By applying this to the integral~\eqref{eq.statphaseintexpan}, we arrive at the following statement.
\begin{theorem}[a variant of {Fedoryuk~\cite[Thm. 2.3]{fedoryuk:pdo1971}}\,] \label{thm.varfedoryuk}
For the integral $I_{\phi}(a)$, under the assumptions and in the notation above, there is a formula
\begin{equation}\label{eq.statphaseanswer}
    I_{\phi}(a)=
    c_{n|2m,\hbar}\, \frac{e^{\frac{i\pi}{4}\sgn d^2\phi(x_0)}}{\sqrt{|\Ber d^2\phi(x_0)|}} \;
    e^{\frac{i}{\hbar}\phi(x_0)}\;
    \left(
     {e^{-\frac{\hbar}{i}\frac{1}{2}d^2\phi(x_0)^{-1}\bigl(\der{}{x}\bigr)}
    \Bigl(e^{\frac{i}{\hbar}\phi^{+}(x;x_0)} a(x)\Bigr)} {\left|\vphantom{\int}\right.}_{x=x_0}\right)\,,
\end{equation}
where the expression in the big brackets is  an expansion in nonnegative powers of $\hbar$ which equals   $a(x_0)\left(1+ O(\hbar)\right)$\ in the lowest order in $\hbar$. 
\end{theorem}
\begin{proof} All what is left to prove is the crucial observation that the result of the application of the operator $L=-\frac{\hbar}{i}\frac{1}{2}d^2\phi(x_0)^{-1}\bigl(\der{}{x}\bigr)$ and its powers to the oscillating function
\begin{equation*}
    u(x)=e^{\frac{i}{\hbar}\phi^{+}(x;x_0)} a(x)\,,
\end{equation*}
evaluated at $x=x_0$, does not contain negative powers of $\hbar$. This is because $\phi^{+}(x;x_0)$ has a zero of order three at $x=x_0$. Indeed, any derivative of order $r$ of the function $u(x)$ is a sum of monomials of the form $a^{(k)}(b')^{k_1}(b'')^{k_2}\dots (b^{(r)})^{k_r}$, where $b(x):=\phi(x;x_0)$, and by $a^{(k)}$, $b'$, $b''$, etc., we mean   partial derivatives in $x$ of degrees $k$, $1$, $2$, etc. We have
\begin{equation*}
    k+k_1+2k_2+3k_3+\dots + rk_r=r
\end{equation*}
and each such monomial carries a factor of $\hbar^{-1}$ in the power
\begin{equation*}
    k_1+ k_2+ k_3+\dots +  k_r
\end{equation*}
arising from differentiating the exponential $e^{\frac{i}{\hbar}b(x)}$\,. Consider $r=2s$ and let $k_1+ k_2+ k_3+\dots +  k_{2s}\geq s$. Then the monomial must contain derivatives $b'$ or $b''$. (If it doesn't, i.e., $k_1=k_2=0$, then $k_3+\dots +  k_{2s}\geq s$ and   the inequality $2s=k+k_1+2k_2+3k_3+\dots + 2sk_{2s}=k + 3k_3+\dots + 2sk_{2s}\geq  k + 3(k_3+\dots +  k_{2s}) \geq 3s$ holds, which is a contradiction.) Since $b'(x_0)=0$ and $b''(x_0)=0$, any partial derivative of $u(x)$ of degree $2s$ at $x=x_0$ may contain $\hbar^{-1}$ only in the powers $<s$. Hence $L^s u(x_0)$, for $s>0$, contains only positive powers of $\hbar$. Also $u(x_0)=a(x_0)$\,. So the expansion is as claimed.
\end{proof}

Now we consider a special case of integrals $I_{\phi}(a)$ where integration is over a $2n|2m$-dimensional space and the phase has the form
\begin{equation}\label{eq.ourphase}
    \phi(y,q)=S(q)-yq+\la g(y)\,.
\end{equation}
Here $\la$ is a formal parameter. The functions $S(q)$ and $g(y)$ may depend on other variables  not shown explicitly. In particular, they may be formal power series in $\hbar$. This type of phase function covers all the examples that we meet in Sections~\ref{sec.quantum}~and~\ref{sec.quantumhomot}: quantum pullback, composition of quantum thick morphisms, transformation of coordinates, and BV-morphisms. Let  $S(q)$ be a formal power series in $q_i$,
\begin{equation}\label{eq.expanofs}
    S(q)=S^0 +\f^iq_i+\frac{1}{2!}S^{ij}q_jq_i +\frac{1}{3!}S^{ijk}q_kq_jq_i +\ldots \
\end{equation}
(while $g(y)$ be a smooth function).
We write $yq$ for $y^iq_i$ and apply similar abbreviations. Integrals we are interested in  have the form
\begin{equation}\label{eq.ourstatphaseint}
    I_{\phi}(a)=\!\!\int\limits_{\,\,\,\,\R{2n|2m}}\!\!\!\! Dy\,\Dbar q\;e^{\frac{i}{\hbar}\phi(y,q)} a(y,q)\,,
\end{equation}
where   $\Dbar q=(2\pi\hbar)^{-n}(i\hbar)^{m}Dq$.  (Note that the factor is  exactly
$c_{2n|2m,\hbar}^{-1}$ in our notation.)

\begin{lemma} \label{lem.statphase}
For the phase $\phi(y,q)$ given by~\eqref{eq.ourphase}, there is a unique stationary point $(y_0,q^0)$, which is the (unique) solution of
the   equations
\begin{equation}\label{eq.statphaseq} 
     y^i= (-1)^{\itt}\der{S}{q_i}(q)\,,  \quad  q_i=\la\der{g}{y^i}\,(y)
\end{equation}
as   perturbation series in $\la$,
\begin{align}
    y^i_0&=\f^i+\la c^i_{(1)}+\frac{\la^2}{2!}c^i_{(2)} + \ldots \ ,   \label{eq.yzero}\\
    q^0_i&=\la \der{g}{y^i}(y_0)=\la \der{g}{y^i}\Bigl(\f +\la c_{(1)}+\frac{\la^2}{2!}c_{(2)} + \ldots \Bigr)\,, \label{eq.qzero}
\end{align}
where the coefficients $c_{(k)}$   are homogeneous polynomials of degrees $k$ in the derivatives of $g$ of orders $\leq k$ at $y=\f$\,, 
\begin{align*}
    c^i_{(1)} =S^{ij}\der{g}{y^j}(\f)\,, \quad
    c^i_{(2)} =S^{ij}S^{kl}\der{g}{y^l}(\f)\dder{g}{y^k}{y^j}(\f) + S^{ijk}\der{g}{y^k}(\f)\der{g}{y^j}(\f)\,, \quad \text{\emph{etc.}}
\end{align*}
The stationary value $\phi(y_0,q^0)=S^0+\la g(\f) +O(\la^2)$\,. The  matrix of $d^2\phi(y_0,q^0)$ is
\begin{equation}\label{eq.matrixQ}
   Q= \begin{pmatrix}
       \la \dder{g}{y^i}{y^j}(y_0) & -(-1)^{\itt}\delta_i{}^j \\
       - \delta^i{}_j  & \dder{S}{q_i}{q_j}(q^0)
     \end{pmatrix}\,.
\end{equation}
Therefore the stationary point $(y_0,q^0)$ is nondegenerate;    we have, for the Hessian,
\begin{equation}\label{eq.hessian}
    |\Ber d^2\phi(y_0,q^0)|=  \Ber \left(\delta_i{}^k-\la \dder{g}{y^i}{y^j}(y_0)\dder{S}{q_j}{q_k}(q^0)\right)=1+O(\la)\,,
\end{equation}
Also, $\sgn d^2\phi(y_0,q^0)= 0$\,.
\end{lemma}
\begin{proof} Equations~\eqref{eq.statphaseq} are 
obtained   by differentiating~\eqref{eq.ourphase}. They combine to give
\begin{equation*}
    y^{i}=(-1)^{\itt}\der{S}{q_i}\left(\la \der{g}{y}(y)\right)\,,
\end{equation*}
solvable by iterations,   giving a unique $(y_0,q^0)$ as in~\eqref{eq.yzero}, \eqref{eq.qzero} with the claimed properties. (Compare~\cite[\S 1]{tv:nonlinearpullback}.)
The expression~\eqref{eq.matrixQ} for the Hesse matrix $Q$ is obtained directly. (For the relevant tensor notation and quadratic forms in the supercase see e.g.~\cite{tv:volumes}.) Its invertibility is  clear from considering it in the zeroth order in $\la$. Equation~\eqref{eq.hessian} for $\Ber Q$ is obtained by multiplying the matrix $Q$ by a matrix
$J=\left(\begin{smallmatrix}0 & -\delta^j{}_k\\
-\delta_j{}^k(-1)^k & 0
\end{smallmatrix}\right)$ with   $\Ber J =\pm 1$\,,
which gives $QJ=\left(\begin{smallmatrix}\delta_i{}^k & -\la g_{ik}\\
-s^{ik}(-1)^{\kt} & \delta^i{}_k
\end{smallmatrix}\right)$,
where $g_{ij}=\dder{g}{y^i}{y^j}(y_0)$, $s^{ij}=\dder{S}{q_i}{q_j}(q^0)$,
and applying to the result the  formula for the Berezinian of a block matrix (analogous to the well-known formula for $\det$). To see that the signature of $Q=d^2\phi(y_0,q^0)$ is zero,   set  $\la=0$ and notice that by a linear change of the variables $y^i$ the   form can be brought to $Q= z^iq_i$.
\end{proof}

Combining Lemma~\ref{lem.statphase} with Theorem~\ref{thm.varfedoryuk}, we immediately obtain the desired statement:
\begin{theorem} \label{thm.ourstatphase}
For $\phi(y,q)= S(q)-yq +\la g(y)$ as in~\eqref{eq.ourphase},  we have
\begin{multline}\label{eq.ourstatphase}
   I_{\phi}(a)=\!\!\int\limits_{\,\,\,\,\R{2n|2m}}\!\!\!\! Dy\,\Dbar q\;e^{\frac{i}{\hbar}\bigl(S(q)-yq+\la g(y)\bigr)} a(y,q)= \\
   e^{\frac{i}{\hbar}\phi(y_0,q^0)}\; b_0^{-\frac{1}{2}} \;
    \left(
     {e^{-\frac{\hbar}{i}\frac{1}{2} L(\der{}{y},\der{}{q})}
    \Bigl(e^{\frac{i}{\hbar}\phi^{+}(y,q;y_0,q^0)} a(y,q)\Bigr)}{\left|\vphantom{\int}\right.}_{y=y_0, q=q^0}\right)\,.
\end{multline}
Here  $(y_0,q^0)$ is the stationary point given by ~\eqref{eq.statphaseq}, \eqref{eq.yzero}, \eqref{eq.qzero}. The function   $\phi^{+}(y,q;y_0,q^0)$ is as above. The matrix $L=Q^{-1}$ is the inverse for  $Q$ given by~\eqref{eq.matrixQ}, so that
\begin{equation*}
    L\Bigl(\der{}{y},\der{}{q}\Bigr)=L^{ij}\der{}{y^i}\der{}{y^j}+2L^{i}{}_j\der{}{q_j}\der{}{y^i} + L_{ij}\der{}{q_j}\der{}{q_i}\,,
\end{equation*}
and $b_0=|\Ber Q|$   given by~\eqref{eq.hessian}. \qed
\end{theorem}

Note that $\phi(y_0,q^0)$ in~\eqref{eq.ourstatphase} has the form $\phi(y_0,q^0)=\phi_0 +O(\hbar)$,
where $\phi_0$ is the stationary phase value for $\phi_0(y,q)$ when $\hbar\to 0$. Also, $b_0=b_{00}+O(\hbar)$, where $b_{00}$ is invertible, hence the Hessian factor can be moved to the phase as a term of order $\geq 1$ in $\hbar$. Finally, since the expression in the big brackets in~\eqref{eq.ourstatphase} has the form $a_0+O(\hbar)$, where $a_0$ is the `classical limit'  of $a(y_0,q^0)$ when $\hbar\to 0$,  we may say that
\begin{equation}
    I_{\phi}(a) = e^{\frac{i}{\hbar}\left(\phi_0 + O(\hbar)\right)} \bigl(a_0+O(\hbar)\bigr)\,.
\end{equation}
In particular, if $a\equiv 1$, then $I_{\phi}(1)= e^{\frac{i}{\hbar}\left(\phi_0 + O(\hbar)\right)}$\,.
From the construction, we also see that both the phase and the amplitude of the integral $I_{\phi}(a)$   are formal power series in $\la$,  which plays the role of a `coupling constant' (if we   borrow  the physicist' term).  We do not use  $\la$ explicitly in the main text, speaking instead of expansions in the powers of the derivatives of the function $g$.


\bigskip

\textsc{Acknowledgements.} The author  thanks Yu.~Yu.~Berest, V.~M.~Buchstaber, A.~Cattaneo, M.~A.~Grigoriev, H.~M.~Khudaverdian, S.~P.~Novikov, A.~V.~Odesskii, A.~G.~Sergeev, M.~A.~Vasiliev, A.~A.~Voronov,  and A.~Weinstein for     discussion. He is  particularly grateful to Yvette~Kosmann-Schwarzbach, Kirill~Mackenzie, and Jim~Stasheff who read earlier versions of the text and made numerous valuable comments.


\begin{thebibliography}{100}
\def\cprime{$'$}

\bibitem{akman:genBV}
F.~Akman.
\newblock On some generalizations of {B}atalin-{V}ilkovisky algebras.
\newblock {\em J. Pure Appl. Algebra}, 120(2):105--141, 1997.

\bibitem{arnold:mathmethodseng}
V.~I. Arnol{\cprime}d.
\newblock {\em Mathematical methods of classical mechanics}, volume~60 of {\em
  Graduate Texts in Mathematics}.
\newblock Springer-Verlag, New York, second edition, 1989.
\newblock Translated from the Russian by K. Vogtmann and A. Weinstein.

\bibitem{bering:higher}
K.~Bering, P.~H. Damgaard, and J.~Alfaro.
\newblock Algebra of higher antibrackets.
\newblock {\em Nuclear Phys. B}, 478(1-2):459--503, 1996.

\bibitem{cattaneo-dherin-weinstein:one}
A.~S. Cattaneo, B. Dherin, and A. Weinstein.
\newblock Symplectic microgeometry {I}: {M}icromorphisms.
\newblock {\em J. Symplectic Geom.}, 8(2):205--223, 2010.

\bibitem{cattaneo-dherin-weinstein:two}
A.~S. Cattaneo, B. Dherin, and A. Weinstein.
\newblock Symplectic microgeometry {II}: {G}enerating functions.
\newblock {\em Bull. Braz. Math. Soc. (N.S.)}, 42(4):507--536, 2011.

\bibitem{cattaneo-dherin-weinstein:three}
A.~S. Cattaneo, B. Dherin, and A. Weinstein.
\newblock Symplectic microgeometry {III}: {M}onoids.
\newblock {\em J. Symplectic Geom.}, 11(3):319--341, 2013.

\bibitem{cattaneo-dherin-weinstein:comorphisms}
A.~S. Cattaneo, B. Dherin, and A. Weinstein.
\newblock Integration of {L}ie algebroid comorphisms.
\newblock {\em Port. Math.}, 70(2):113--144, 2013.


\bibitem{egorov:canonicpdo1969}
Ju.~V. Egorov.
\newblock The canonical transformations of pseudodifferential operators.
\newblock {\em Uspehi Mat. Nauk}, 24(5 (149)):235--236, 1969.

\bibitem{egorov:canonicpdo1971}
Ju.~V. Egorov.
\newblock Canonical transformations and pseudodifferential operators.
\newblock {\em Trudy Moskov. Mat. Ob\v s\v c.}, 24:3--28, 1971.

\bibitem{egorov:microlocalencyclo}
Yu.~V. Egorov.
\newblock Microlocal analysis.
\newblock {\em Partial differential equations. {IV}}, volume~33 of {\em
  Encyclopaedia of Mathematical Sciences}. M.~A. Shubin and Yu.~V. Egorov, editors.
\newblock Springer-Verlag, Berlin, 1993.
\newblock Translated from the Russian by P. C. Sinha.


\bibitem{fedoryuk:pdo1971}
M.~V. Fedorjuk.
\newblock Method of stationary phase, and pseudodifferential operators.
\newblock {\em Uspehi Mat. Nauk}, 26(1(157)):67--112, 1971.
\newblock Russian Math. Surveys 26 (1971), no. 1, 65--115.

\bibitem{fock:canon1959-69}
V.~A.~Fock.
\newblock On the canonical transformation in classical and quantum mechanics.
\newblock {\em Acta Phys. Acad. Sci. Hungar.}, 27:219--224, 1969.
(First published in Russian in the Bulletin (Viestnik) of the Leningrad University,
\textnumero~16, p. 67, 1959.)

\bibitem{fock:book}
V.~A. Fock.
\newblock {\em Fundamentals of quantum mechanics}.
\newblock Mir Publishers, Moscow, 1978.
\newblock Translated from the Russian.

\bibitem{groth:ega4}
A.~Grothendieck.
\newblock \'{E}l\'ements de g\'eom\'etrie alg\'ebrique. {IV}. \'{E}tude locale
  des sch\'emas et des morphismes de sch\'emas {IV}.
\newblock {\em Inst. Hautes \'Etudes Sci. Publ. Math.}, (32):361, 1967.


\bibitem{gs:geomasymp}
V.~Guillemin and S.~Sternberg.
\newblock {\em Geometric asymptotics}.
\newblock American Mathematical Society, Providence, R.I., 1977.
\newblock Mathematical Surveys, No. 14.

\bibitem{mackenzie_and_higgins:duality}
Ph.~J. Higgins and K.~C.~H. Mackenzie.
\newblock Duality for base-changing morphisms of vector bundles, modules, {L}ie
  algebroids and {P}oisson structures.
\newblock {\em Math. Proc. Cambridge Philos. Soc.}, 114(3):471--488, 1993.


\bibitem{hoer:spectralofell-1968}
L. H\"ormander.
\newblock The spectral function of an elliptic operator.
\newblock {\em Acta Math.}, 121:193--218, 1968.

\bibitem{hoer:fio1-1971}
L. H\"ormander.
\newblock Fourier integral operators. {I}.
\newblock {\em Acta Math.}, 127(1-2):79--183, 1971.


\bibitem{tv:laplace1}
H.~M.~Khudaverdian and Th.~Th.~Voronov.
\newblock On odd {Laplace} operators.
\newblock {\em Lett. Math. Phys.}, 62:127--142, 2002.

\bibitem{tv:laplace2}
H.~M.~Khudaverdian and Th.~Th.~Voronov.
\newblock On odd {Laplace} operators. {II}.
\newblock In V.~M. Buchstaber and I.~M. Krichever, eds., {\em Geometry,
  Topology and Mathematical Physics. {S.~P.~Novikov's} seminar: 2002--2003},
  volume 212 of {\em Amer. Math. Soc. Transl. (2)}, pages 179--205. Amer. Math.
  Soc., Providence, RI, 2004.

\bibitem{tv:higherpoisson}
H.~M. Khudaverdian and Th.~Th. Voronov.
\newblock Higher {Poisson} brackets and differential forms.
\newblock In {\em X{XVII} {W}orkshop on {G}eometrical {M}ethods in {P}hysics},
  vol. 1079 of {\em AIP Conf. Proc.}, pages 203--215. Amer. Inst. Phys.,
  Melville, NY, 2008.


\bibitem{tv:linfbialg}
H.~M. Khudaverdian and Th.~Th. Voronov.
\newblock {Thick morphisms and  higher Koszul brackets.} 
\newblock (In preparation).


\bibitem{kontsevich:quant}
M.~Kontsevich.
\newblock Deformation quantization of {Poisson} manifolds, {I}.
\newblock {\tt math.QA/9709180}. 
Also: {\em Lett. Math. Phys.}, 66(3):157--216, 2003.

\bibitem{koszul:crochet85}
J.-L.~Koszul.
\newblock Crochet de {S}chouten-{N}ijenhuis et cohomologie.
\newblock {\em Ast\'erisque}, (Numero Hors Serie):257--271, 1985.
\newblock The mathematical heritage of \'Elie Cartan (Lyon, 1984).

\bibitem{lada:stasheff}
T.~ Lada and J.~Stasheff.
\newblock Introduction to {SH} {L}ie algebras for physicists.
\newblock {\em Internat. J. Theoret. Phys.}, 32(7):1087--1103, 1993.

\bibitem{mackenzie:diffeomorphisms}
K.~C.~H. Mackenzie.
\newblock On certain canonical diffeomorphisms in symplectic and {P}oisson
  geometry.
\newblock In {\em Quantization, Poisson brackets and beyond (Manchester,
  2001)}, vol. 315 of {\em Contemp. Math.}, pp. 187--198. Amer. Math. Soc.,
  Providence, RI, 2002.

\bibitem{mackenzie:book2005}
K.~C.~H. Mackenzie.
\newblock {\em General theory of {L}ie groupoids and {L}ie algebroids}, volume
  213 of {\em London Mathematical Society Lecture Note Series}.
\newblock Cambridge University Press, Cambridge, 2005.

\bibitem{mackenzie:bialg}
K.~C.~H. Mackenzie and P.~Xu.
\newblock Lie bialgebroids and {P}oisson groupoids.
\newblock {\em Duke Math. J.}, 73(2):415--452, 1994.

\bibitem{maslov:pert1965}
V.~P. Maslov.
\newblock {\em Perturbation theory and asymptotic methods}.
\newblock Moscow State University Publ., Moscow, 1965.
\newblock Russian.


\bibitem{shaf:a}
I.~R. Shafarevich.
\newblock {\em Algebra I. Basic Notions of Algebra}, vol.~11 of {\em
  Encyclopaedia of Mathematical Sciences}.
\newblock Springer-Verlag, Berlin etc., 1990.
\newblock Transl. from the Russian   by M. Reid.

\bibitem{shubin:book}
M.~A. Shubin.
\newblock {\em Pseudodifferential operators and spectral theory}.
\newblock Springer-Verlag, Berlin, second edition, 2001.
\newblock Transl. from the 1978 Russian original by Stig I. Andersson.

\bibitem{treves:introtofio}
F.~Tr\`eves.
\newblock {\em Introduction to pseudodifferential and {F}ourier integral
  operators. {V}ol. 2}.
\newblock Plenum Press, New York-London, 1980.


\bibitem{tulczyjew:1977}
W.~M. Tulczyjew.
\newblock A symplectic formulation of particle dynamics.
\newblock In {\em Differential Geometric Methods in Mathematical Physics, Bonn
  1975}, volume 570 of {\em Lecture Notes in Math.}, pages 457--463.
  Springer-Verlag, Berlin, 1977.

\bibitem{vishik:eskin-conv1965}
M.~I. Vi{\v s}ik and G.~I. \`Eskin.
\newblock Convolution equations in a bounded region.
\newblock {\em Uspehi Mat. Nauk}, 20(3 (123)):89--152, 1965.


\bibitem{tv:class}
Th.~Th. Voronov.
\newblock Class of integral transforms induced by morphisms of vector bundles.
\newblock {\em Matemat. Zametki}, 44(6):735--749, 1988.

\bibitem{tv:graded}
Th. Th. Voronov.
\newblock Graded manifolds and {Drinfeld} doubles for {Lie} bialgebroids.
\newblock In  {\em Quantization, Poisson Brackets and
  Beyond}, vol. 315 of {\em Contemp. Math.}, pp. 131--168. Amer. Math.
  Soc., Providence, RI, 2002.

\bibitem{tv:higherder}
Th. Th. Voronov.
\newblock Higher derived brackets and homotopy algebras.
\newblock {\em J. of Pure and Appl. Algebra}, 202(1--3):133--153, 2005.



\bibitem{tv:napl}
Th. Th. Voronov.
\newblock On a non-{Abelian} {Poincar\'{e}} lemma.
\newblock {\em Proc. Amer. Math. Soc.}, 140:2855--2872, 2012.

\bibitem{tv:qman}
Th. Th. Voronov.
\newblock {$Q$-manifolds} and higher analogs of {Lie} algebroids.
\newblock In {\em XXIX {W}orkshop on {G}eometric {M}ethods in {P}hysics},
  vol. 1307 of {\em AIP Conf. Proc.}, pp. 191--202. Amer. Inst. Phys.,
  Melville, NY, 2010.

\bibitem{tv:qman-mack}
Th. Th. Voronov.
\newblock {$Q$}-manifolds and {Mackenzie} theory.
\newblock {\em Comm. Math. Phys.}, 315(2):279--310, 2012.

\bibitem{tv:nonlinearpullback}
Th.~Th. Voronov.
\newblock ``Nonlinear pullbacks" of functions and $L_{\infty}$-morphisms for homotopy Poisson structures.
\newblock {\em J. Geom. Phys.}, 111: 94--110, 2017.



\bibitem{tv:oscil}
Th.~Th. Voronov.
\newblock Thick morphisms of supermanifolds and oscillatory integral operators.
\newblock {\em Russian Math. Surveys}, 71(4):784--786, 2016.

\bibitem{tv:qumicro}
Th.~Th. Voronov.
\newblock Quantum microformal morphisms of supermanifolds: an explicit formula
  and further properties.
\newblock \texttt{arXiv:1512.04163 [math-ph]}.

\bibitem{tv:volumes}
Th.~Th. Voronov.
\newblock On volumes of classical supermanifolds.
\newblock {\em Sbornik: Mathematics}, 207(11):1512--1536, 2016.


 \bibitem{tv:homotopytriang}
Th.~Th. Voronov.
\newblock {$L_{\infty}$-bialgebroids and homotopy Poisson structures.}
\newblock (In preparation).



\bibitem{weinstein:bull69}
A.~Weinstein.
\newblock Symplectic structures on {B}anach manifolds.
\newblock {\em Bull. Amer. Math. Soc.}, 75:1040--1041, 1969.

\bibitem{weinstein:sympl71}
A.~Weinstein.
\newblock Symplectic manifolds and their {L}agrangian submanifolds.
\newblock {\em Advances in Math.}, 6:329--346, 1971.

\bibitem{weinstein:symplcat82}
A.~Weinstein.
\newblock The symplectic ``category''.
\newblock In {\em Differential geometric methods in mathematical physics
  ({C}lausthal, 1980)}, vol. 905 of {\em Lecture Notes in Math.}, pp.
  45--51. Springer, Berlin--New York, 1982.

\bibitem{weinstein:symplectic-geometry1981}
A.~Weinstein.
\newblock Symplectic geometry.
\newblock {\em Bull. Amer. Math. Soc. (N.S.)}, 5(1):1--13, 1981.


\bibitem{weinstein:coisotropic88}
A.~Weinstein.
\newblock Coisotropic calculus and {P}oisson groupoids.
\newblock {\em J. Math. Soc. Japan}, 40(4):705--727, 1988.

\bibitem{weinstein:symplectic-geometries}
A.~Weinstein.
\newblock Symplectic categories.
\newblock {\em Port. Math.}, 67(2):261--278, 2010.

\bibitem{weinstein:wehrheimwoodwardcat}
A.~Weinstein.
\newblock A note on the {W}ehrheim-{W}oodward category.
\newblock {\em J. Geom. Mech.}, 3(4):507--515, 2011.


\end{thebibliography}

\end{document}